\documentclass[12pt,article]{amsart}

\usepackage[margin=1in]{geometry}
\usepackage{amsmath, amssymb,amsthm,url,mathrsfs,graphicx,amscd,amsfonts}
\usepackage[mathscr]{eucal}
\usepackage{dsfont}
\usepackage{mathtools}
\usepackage{hyperref}
\usepackage{bbm}
\usepackage{wrapfig}
\usepackage[utf8]{inputenc}
\usepackage{graphicx,todonotes,hyperref,blindtext}
\usepackage{epstopdf}
\usepackage{inputenc}
\usepackage{enumitem}
\usepackage{amsmath}
\usepackage{graphicx,todonotes,hyperref}
\usepackage{xcolor}
\usepackage{amssymb}
\newtheorem*{acknowledgements*}{Acknowledgements}

\renewcommand{\Im}{\operatorname{Im}}
\renewcommand{\Re}{\operatorname{Re}}

\allowdisplaybreaks
\newtheorem{theorem}{Theorem}[section]
\newtheorem{lemma}[theorem]{Lemma}
\newtheorem{corollary}[theorem]{Corollary}
\newtheorem{proposition}[theorem]{Proposition}

\newtheorem*{theorem*}{Theorem}
\theoremstyle{remark}\newtheorem{remark}{Remark}

\newtheorem{case}{Case}

\numberwithin{equation}{section}

\renewcommand{\pmod}[1]{\left(\mathrm{mod}\,#1\right)}
\renewcommand{\phi}{\varphi}

\def\Var{\operatorname{Var}}

\begin{document}

\title[Variance of a general class of multiplicative functions]{The variance of a general class of multiplicative functions in short intervals}

\author[P. Darbar]{Pranendu Darbar}
\address[Pranendu Darbar]{School of Mathematics and Statistics \\ University of New South
	Wales, Sydney NSW 2052, Australia}
\email{darbarpranendu100@gmail.com}

\author[M. K. Das]{Mithun Kumar Das}
\address[Mithun Kumar Das]{School of Mathematical Sciences, National Institute of Science Education and Research \\ A CI of Homi Bhabha National Institute, Jatni, Khurda, 752050, India}
\email{das.mithun3@gmail.com}

%\author{\textcolor{red}{Advika}}

\begin{abstract}
We study a general class of multiplicative functions by establishing a connection between their ``short averages" and ``long average". More precisely, we employ Fourier analysis and the counting of rational points on specific binary forms to provide asymptotic estimates for the variance of this class within short intervals.
   Our results apply to notable multiplicative functions such as $\mu_k(n)$, $\frac{\phi(n)}{n}$, $\sigma_{\alpha}(n)$, among others, yielding several new results and improvements in the realm of short interval analysis. Remarkably, our results disprove a conjecture of van Overbeeke concerning the variance of $\frac{\phi(n)}{n}$.
\end{abstract}

\maketitle

\section{Introduction}\label{section 1}
Multiplicative functions, which satisfy the condition $f(mn) = f(m)f(n)$ for all coprime $(m, n)$, are among the most fascinating arithmetical functions for number theorists. In analytic number theory, a key focus revolves around the asymptotic evaluation of the global partial sum
$
\sum_{n \leq X} f(n)
$
for any multiplicative function $f$. While the global behavior of multiplicative functions has been extensively studied, understanding these partial sums in short intervals presents a formidable challenge. 
It becomes imperative to examine the behavior of multiplicative functions within short intervals and compare them to their global partial sums. To be more precise, let $x$ be a randomly selected number from the interval $[0, X]$, and let $H:=H(X)$ be a function of $X$. The objective is to understand the asymptotic formula for the sum $
\sum_{x\leq n\leq x+H}f(n)$
while imposing suitable growth conditions on the interval of length $H$. Instead of establishing results for each individual short interval, which often proves to be challenging, an alternative approach involves characterizing the fluctuations in these sums through the analysis of their variances
\[
\Var_f(X; H):=\frac{1}{X}\int_{X}^{2X}\bigg|\sum_{x<n\leq x+H}f(n)- \frac{H}{X}\sum_{n\leq X}f(n)\bigg|^2 dx.
\]
%which is the second moment of $Z_f(x):=\sum_{x<n\leq x+H}f(n)- \frac{H}{X}\sum_{n\leq X}f(n)$.

Selberg \cite{SEL} initially examined the variance $\Var_{\Lambda}(X; H)$ with the aim of enumerating primes within short intervals, focusing on the von Mangoldt function $\Lambda$. Nonetheless, obtaining the asymptotic formula for $\Var_{\Lambda}(X; H)$ poses significant challenges due to its direct connection to the Riemann hypothesis (RH) and the pair correlation conjecture for the Riemann zeta function (refer to \cite{GM}). 
 Furthermore, in \cite{GC} and \cite{GM}, conjectures were formulated concerning $\Var_{\mu}(X; H)$ of the M\"{o}bius function and $\Var_{\Lambda}(X; H)$ respectively. Similarly, a comparable conjecture can be proposed for the $k$-fold divisor function $d_k(n)$ in \cite{KRRR}. These conjectures exhibit a close connection to the observation that the zeros of families of $L$-functions tend to exhibit a distribution resembling the eigenvalues of specific random matrices, as explored in the work of Katz and Sarnak \cite{KS}. Moreover, by examining the asymptotic behavior of higher moments, including $\Var_f(X; H)$ (representing a second moment), one can derive distributional results for the arithmetical function $f$.

In recent years, significant progress has been made in establishing function field analogues of these conjectures, following the pioneering work of Keating and Rudnick \cite{KR}. The primary approach involves expressing the variance for functions of coefficients in terms of suitable $L$-functions and applying equidistribution results to establish the transformation of these sums into matrix integrals through the $q$-limit (for more details, refer to \cite{KR, KR2, KRRR}). Notably, the theorems concerning $\mathbb{F}_q[x]$ have sparked new conjectures for a variety of intriguing arithmetical functions in the number field setting.

Lastly, a favorable estimate for the quantity $\sum_{x<n\leq x+H}f(n)$ can be derived by using a non-trivial upper bound on $\Var_f(X;
 H)$ along with the application of Chebyshev's inequality. Prior to presenting our main results, we provide an overview of the advancements made in this area during the past decade.

\subsubsection*{Work of Matom\"{a}ki, Radziwi{\l}{\l} and Tao} In a groundbreaking achievement, Matom\"{a}ki and Radziwi{\l}{\l} \cite{MR} established that for every multiplicative function $f:\mathbb{N} \to [-1, 1]$, the short-term average closely approximates its long-term average for almost all $x \leq X$, in the sense that
 \[
\Var_f(X; H)=o(H^2), \quad \text{ as }\, H = H(X) \to \infty.
\]
\noindent
This result has numerous implications, including in evaluating the average behavior of the Liouville function, estimating the count of smooth numbers within short intervals, and determining the number of sign changes of the Liouville function up to a certain threshold $X$. Additionally, Matom\"{a}ki, Radziwi{\l}{\l}, and Tao [\cite{MRT}, Theorem A.1] have generalized this result to encompass any non-pretentious multiplicative function $f:\mathbb{N}\to \mathbb{C}$ bounded by $1$.

\vspace{2mm}
For the M\"{o}bius function, it is expected that $\Var_{\mu}(X; H)\sim H/\zeta(2).$ Furthermore, it is believed that $\sqrt{\zeta(2)/H}\sum_{x<n\leq x+H}\mu(n)$ follows approximately a normal distribution. Ng \cite{NG} explored this conjecture by employing the generalized Riemann hypothesis and a strong version of Chowla's conjecture on correlations of the M\"{o}bius function. For $H\leq X^{1/4-\epsilon}$, he established the validity of the aforementioned variance, and assuming these conjectures, a Gaussian distribution was demonstrated for $H\leq X^{\epsilon}$. It is crucial to impose the constraint $H\leq X^{1-\epsilon}$, as exceeding this limit would yield non-Gaussian statistics. In the context of function fields, Keating and Rudnick [\cite{KR}, Theorem 1.2] have proven the analogous result in the limit of large finite fields ($q\to \infty$).
\subsubsection*{Work of Gorodetsky, Matom\"{a}ki, Radziwi{\l}{\l} and Rodgers}
Motivated by the work of Keating and Rudnick \cite{KR}, in a recent study by Gorodetsky et al. \cite{GMRR}, the behavior of square-free numbers in short intervals was investigated. They demonstrated that for any given $\epsilon\in (0, 1/100)$ and $2\leq H\leq X^{6/11-\epsilon}$, the following result holds:
\begin{align}\label{square-free case}
\Var_{\mu^2}(X; H)=cH^{\frac12}+O\left(H^{\frac12-\frac{\epsilon}{16}}\right),
\end{align}
where $c>0$ is a specific constant.
Keating and Rudnick \cite{KR} established the optimal validity of \eqref{square-free case} over $\mathbb{F}_q[x]$ as $q\to \infty$. They proposed that \eqref{square-free case} holds true for $H\leq X^{1-\epsilon}$. Also, by employing a method introduced by Hall \cite{Hall}, they were able to establish an analogous expression for \eqref{square-free case} over $\mathbb{F}_q[x]$ in the limit of large degrees ($n\to \infty$) with a narrower range for $H$.  Furthermore, Gorodetsky et al. \cite{GMR} showed that the counts of square-free integers up to $X$ in short intervals of size $H$ tend to
a Gaussian distribution as long as $H\to \infty$ and $H=X^{o(1)}$. This was achieved by computing an asymptotic formula for the higher moments of $\sum_{x<n \leq x+H}\mu^2(n)-{H}/{\zeta(2)}$.

Note that the number of square-free integers in almost all intervals exhibits a cancellation of $H^{3/4}$ relative to its mean, in contrast to the work of Matom\"{a}ki, Radziwi{\l}{\l}, and Tao \cite{MRT}, where the cancellation for the Liouville or M\"{o}bius function is $o(1)$ relative to its mean. However, they achieve a wider range of applicability for shorter intervals. This distinction arises from the inherent rigidity of sequences composed of square-free integers, while the sequence of the Liouville or M\"{o}bius function lacks such rigidity.

An important motivation for studying the asymptotics of variance, beyond just its upper bounds, lies in the quest to establish a probabilistic model for fundamental arithmetical functions such as the M\"{o}bius function, prime counting function (in conjunction with $\Lambda(n)$), divisor functions $d_k(n)$, and their interrelationships with the distribution of zeros of the corresponding $L$-functions.

 Inspired by the preceding works, this paper is driven by the aim of exploring the asymptotic formula for $\Var_f(X; H)$ within a general class of multiplicative functions outlined in Section 2. By doing so, we are able to derive the variance for intriguing multiplicative functions, including the indicator function of $k$-free integers, the normalized Euler totient function $\frac{\phi(n)}{n}$, and the generalized sum of divisor functions. These functions are analyzed comprehensively in the subsequent subsections. We conclude with a discussion of the variances of several lesser-known multiplicative functions in the appendix.
 \subsection{The variance of $k$-free integers}
 For an integer $k \geq 2$, we say an integer $n$ is $k$-free if it has no divisor $d > 1$ which
 is a perfect $k$-th power. It is well known that the $k$-free integers have density $\frac{1}{\zeta(k)}$ in $\mathbb{N}$.
 Let
 \[
 E_k(x) :=\sum_{n\leq x}\mu_k(n)-\frac{x}{\zeta(k)}.
 \]
The currently most favorable upper bound, credited to Walfisz \cite{Wal}, enjoys multiple refinements for various values of $k$, assuming the validity of the Riemann hypothesis (see \cite{BP, GP, Liu, Liu2}). It is widely conjectured that
 \begin{align}\label{conjecture on k-free integers}
 E_k(x)\ll x^{\frac{1}{2k}+ \epsilon}.
 \end{align}
 This leads to significant implications such as the verification of the Riemann hypothesis and the simplicity of zeros of the Riemann zeta function along the critical line. Recent research by Mossinghoff et al. \cite{Mos} has further advanced our understanding in this area. They have demonstrated that infinitely often, for $k=2, 3, 4,$ and $5$, the ratio $E_k(x)/x^{1/2k}$ satisfies $E_k(x)/x^{1/2k}<-3$ as well as $E_k(x)/x^{1/2k}>3$. In addition, they studied the ratio $E_k(x)/x^{1/2k}$ for sufficiently large values of $k$.
Drawing a parallel to the conjecture presented in \eqref{conjecture on k-free integers}, it is reasonable to propose that, for any given $\epsilon>0$, and within the range $x^{\epsilon} \leq H \leq x$, the following statement holds true:
 \begin{small}
 \begin{align}\label{conjecture on k free integers in short intervals}
 \sum_{x<n \leq x+H}\mu_k(n)=\frac{H}{\zeta(k)}+ O\left(H^{\frac{1}{2k}+\epsilon}\right).
 \end{align}
 \end{small}
% When $H$ is close to $x^{\epsilon}$, there are no asymptotic estimate known (see the work of Tolev \cite{Tolev}). For large $H$, say $H =x$, estimating the sum in \eqref{conjecture on k free integers in short intervals} asymptotically is straightforward, but obtaining an error term
 %$O\left(H^{\frac{1}{2k}+\epsilon}\right)$ is an open problem, even conditionally on the Riemann hypothesis. %Observe that the conjectures \eqref{conjecture on k free integers in short intervals} also imply the Riemann Hypothesis. \\

In the scenario where $k=2$, Gorodetsky et al. \cite{GMRR} made significant advancements over Hall's work \cite{Hall} by employing Fourier analysis and integer point-counting techniques associated with binary quadratic forms. They successfully established \eqref{conjecture on k free integers in short intervals} on average with a more precise error term. Subsequently, Gorodetsky, Mangerel, and Rodgers \cite{GMR} computed moments for $k$-free integers and the more general class of $B$-free numbers in order to examine their Gaussian distribution. Notably, the second moment (\cite{GMR}, Propositions 1.7 and 1.9) yields an error term of magnitude $O(H^{\frac{1}{k}-\epsilon})$ under the condition that $H\leq X^{\frac{k(k-1)}{(k+1)(2k-1)}-\epsilon}$. By utilizing Theorem \ref{main theorem for simultaneouly k-free intergers}, we can derive an application that improves upon the aforementioned result presented in \cite{GMR}. This improvement is applicable when considering $k\geq 3$ and extends the range of $\epsilon$ compared to Theorem $1$ of \cite{GMRR} which specifically addressed the case of square-free numbers.

 \begin{corollary}{\label{k-free case}}
 	Let $X \geq 1$ be sufficiently large. Suppose $e(k, 0)$ and $g(k)$ are defined in  \eqref{10/5/14:18}. Let $\epsilon\in (0, \frac{3}{10})$ be given.  For $2\leq H\leq X^{e(k, 0)-\epsilon}$, we have
 	\begin{align*}
 	\Var_{\mu_k}(X; H) =  c_{ k}H^{\frac{1}{k}} + O\left(H^{\frac{1}{k}- \frac{\epsilon}{3(k+1)}}\right),
 	\end{align*}
 	where 
 	\begin{small}
 	\begin{equation*}
 	c_k=  -\frac{1}{2\pi^2}\chi{\Big(\frac{k+1}{k}\Big)}\zeta{\Big(2-\frac{1}{k}\Big)}\prod_{p}\bigg( 1-\frac{1}{p^2} - \frac{2(p-1)}{p^{k+1}}\bigg)
 	\end{equation*}
 	\end{small}
 and $\chi(s)$ is defined by \eqref{10/5/00:43}.	Assuming the Lindel\"{o}f hypothesis, the above estimate holds for a wider range $H\leq X^{g(k)-\epsilon}$.
 \end{corollary}
 
 \begin{remark}
 	We have made significant improvements in expanding the length of short intervals for $k\geq 3$. For instance, $e(3,0)\approx 0.60537,\, e(4, 0)\approx 0.65797, \, \ldots , e(10^5, 0)\approx 0.77346$ in comparison with the previous values obtained from \cite[Proposition 1.7]{GMR} as $e(3,0)\approx 0.3, e(4, 0)\approx 0.34286,\ldots , e(10^5, 0)\approx 0.49999$. Under the Lindel\"{o}f hypothesis, we are able to further extend these exponents denoted by $g(k)$: $g(3)\approx 0.7182, g(4)\approx 0.7355, \ldots, g(10^5)\approx 0.77346$. For sufficiently large values of $k$, the exponents remain the same, rounded to some decimal places, regardless of the validity of the Lindel\"{o}f hypothesis. 
 \end{remark}
 \begin{remark}
 It is possible to extend the length of a short interval even further, allowing $H\leq X^{1-\epsilon}$, with a slightly weaker estimate given by $\Var_{\mu_k}(X; H)\ll H^{1/k +\epsilon}$ under RH. For a more comprehensive understanding, refer to Remark \ref{under RH}. 
 \end{remark}

\subsection{The variance of Euler $\phi$-function in short intervals}
Here, we examine the variance of the normalized Euler totient function over integers. van Overbeeke [\cite{over}, Theorem 6.1] obtained an asymptotic formula for the normalized Euler phi function over $\mathbb{F}_q[x]$ in the $q$-limit by computing an integral over the group of unitary matrices. Additionally, he predicted \cite[Conjecture 2.1]{over} that, in the case of integers, the discrete variance of $\frac{\phi(n)}{n}$ converges to a fixed constant $\frac{1}{6 \zeta(2)}-\frac{1}{6 \zeta(2)^2}=0.03\ldots$ within a short interval of size $H=\Theta(x^{\delta})$, where $0<\delta \leq 1$. However, we present a result concerning the variance of $\frac{\phi(n)}{n}$ within a wider short interval. This is an application of Theorem \ref{main theorem for lower growth rate}. 
 \begin{corollary}\label{main theorem of euler toitient}
 	Let $\epsilon>0$ be given. For $2\leq H\leq X^{1-\epsilon}$, we have 
 	\begin{equation*}
 	\Var_{\frac{\phi(n)}{n}}(X; H) = c_{\zeta}(H) +O(H^{-\epsilon}),
 	\end{equation*}  
 	where
 	\begin{small}
 	\[
 	c_{\zeta}(H)=\sum_{\substack{d=1}}^{\infty}\frac{\mu^2(d)}{d^2}\prod_{p\nmid d}\left(1 + \frac{2}{p^2}\right)\left( \left\{ \frac{H}{d} \right\} - \left\{\frac{H}{d} \right\} ^2\right).
 	\]
 \end{small}
 \end{corollary}
 
 \begin{remark}
We have observed that $c_{\zeta}(H)\leq 0.833$. By using the computer software ``Mathematica", we have created graphical representations (refer to Figure \ref{fluctuation of phi}) that depict the fluctuations of the constant $c_{\zeta}(H)$ in relation to varying values of $H$.

Of particular interest is the third graph, which reveals that for integer values $2\leq H\leq 10000$, the range of $c_{\zeta}(H)$ spans from $0.01$ to $0.20$. In contrast, the first two graphs relate to the scenario in which $H$ is treated as a real number. The second graph effectively illustrates the smooth transition of the constant between consecutive integers. Additionally, the third graph demonstrates the distinct separation of layers based on even and odd integers.
\end{remark}
\begin{remark}
Numerically, even integers are represented in the first and second layers (from the bottom to the top in the third picture), while odd integers occupy the third and fourth layers. Specifically, the second layer contains a higher density of even integers due to the distribution of sequences like $\{H=6n+2\}_{n\in \mathbb{N}}$ and $\{H=6n+4\}_{n\in \mathbb{N}}$. The remaining $\{H=6n\}_{n\in \mathbb{N}}$ is distributed along the first layer. Also, the fourth layer exhibits a greater density of odd integers due to the presence of $\{H=6n+1\}_{n\in \mathbb{N}}$ and $\{H=6n+5\}_{n\in \mathbb{N}}$, and the third layer is occupied by $\{H=6n+3\}_{n\in \mathbb{N}}$.
\end{remark}

A key insight is that as $H$ approaches infinity, the variance limit converges to a specific constant that is determined by $H$, in contrast to the absolute constant proposed by van Overbeeke. Notably, by considering a subsequence $\{H=p\}_{p\in \mathcal{P}}$, we establish a lower bound for the constant $c_{\zeta}(H)$ as $H$ tends to infinity. This discrepancy exposes an inaccuracy in van Overbeeke's conjecture.
\begin{theorem}\label{lower bound for euler phi}
Let $\{H=p\}_{p\in \mathcal{P}}$ be the sequence of prime numbers. Then 
	\[
	c_{\zeta}(H)> 0.1036\ldots+o(1), \quad \text{ as } H\to \infty.
	\]
\end{theorem}

\begin{figure}
	\includegraphics[scale=0.43]{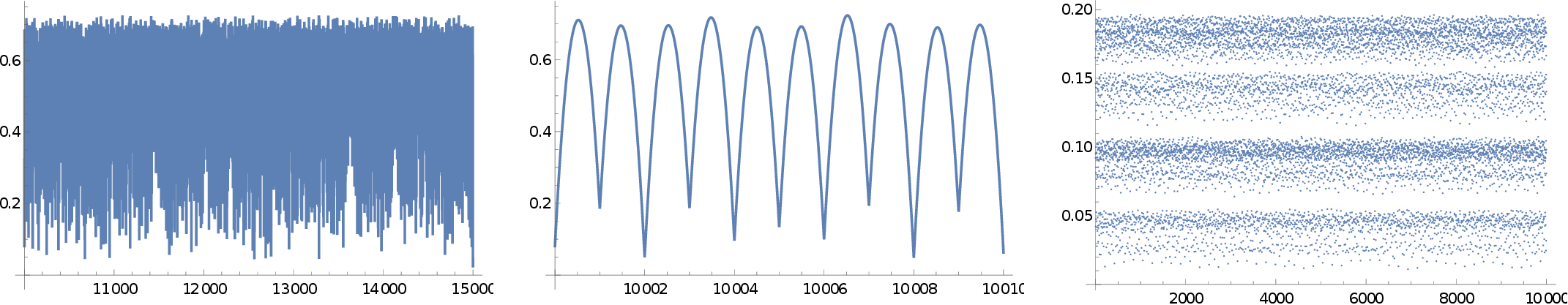}
	\caption{Fluctuation of $c_{\zeta}(H)$}\label{fluctuation of phi}
	%	\label{figure_outlier}
	%\includegraphics[scale=0.8]{even-odd-all.pdf}
	%%\includegraphics[scale=0.6]{plot1.pdf}  \quad 
	%%\includegraphics[scale=0.6]{plot2.pdf}
	%%\caption{ asmfekng}
	%%	\label{figure_outlier}
\end{figure} 

\subsection{The variance of generalized sum of divisor function}
Let $\alpha\in \mathbb{R}$. The $n$-th coefficient of the Dirichlet series $\zeta(s)\zeta(s - \alpha)$ represents the generalized sum of divisor function $\sigma_{\alpha}(n)$, which is defined by 
\[
\sigma_{\alpha}(n)=\sum_{d\mid n}d^{\alpha}.
\]
Chowla \cite{Chowla} initially studied the variance of $\sigma_{\alpha}(n)$ and established its behavior for the range $-1<\alpha<-\frac12$,
\[
\frac{1}{X}\int_{X}^{2X}\bigg|\sum_{n\leq x}\sigma_{\alpha}(n)-\zeta(1-\alpha)x\bigg|^2 dx=c_1+O\left(X^{1/2+\alpha} \log X \right),
\] 
for some explicit constant $c_1$ depending on $\alpha$. Afterward, Lau \cite {Lau} established an $\Omega_{\pm}$-result. In \cite[Theorem (eq. (3))]{KT}, Kiuchi and Tanigawa  obtained that for $-1<\alpha<-\frac12$ and $H\ll X^{1/2}$,
\begin{align}\label{result of kiuchi}
\Var_{\sigma_{\alpha}}(X; H) \ll X^{\epsilon}.
\end{align}
As a consequence of Theorem \ref{main theorem for lower growth rate}, our next corollary yields a significant improvement of Kiuchi and Tanigawa's aforementioned result across a broader range of $H$ and $\alpha$.
\begin{corollary}\label{divisor functions}
	Let $\epsilon>0$ be given. Then for $2\leq H\leq X^{\min\big\{\frac{29}{113+84\alpha}, \frac{-8\alpha - 1}{17+16\alpha}\big\}-\epsilon}$, whenever $\alpha\in (-1, -1/2)$ and $2\leq H\leq X^{1-\epsilon},$ when $\alpha\in (-2, -1]$, we have 
	\begin{equation*}
\Var_{\sigma_{\alpha}}(X; H) =
	c(H) +O(H^{-\epsilon}),
	\end{equation*}  
	where
	\begin{small}
	\[
	c(H)=\frac{\zeta(1-\alpha)^2}{\zeta(2-2\alpha)}\sum_{d=1}^{\infty}d^{2\alpha}\left( \left\{\frac{H}{d} \right\} - \left\{\frac{H}{d} \right\} ^2\right).
	\]
	\end{small}
	Assuming Lindel\"{o}f hypothesis, we have the extended range $H\leq X^{\frac{1}{3+2\alpha}-\epsilon}$ for $\alpha\in (-1, -1/2)$.
\end{corollary}

\begin{remark}
 First of all, Corollary \ref{divisor functions} presents an asymptotic formula featuring a constant in its main term, thereby enhancing the previously established bound in \eqref{result of kiuchi}. Additionally, we provide an extended range for the parameter $H$, specifically for $\alpha\in (-1, -0.66]$, effectively surpassing the $\frac{1}{2}$-barrier, although for $\alpha\in (-0.66, -1/2)$ we can still have asymptotic formula with weaker range of $H$ compared to Kiuchi and Tanigawa.  For example, when $\alpha=-\frac{3}{4}$, the exponent becomes $0.58$, and for $\alpha=-0.99$, it is approximately $0.97$, while for $\alpha=-0.55$, we get the exponent $\approx 0.41$. 
	
	%Here we can not extend the range further even if under Riemann  hypothesis to our desired $H\leq X^{1-\epsilon}$ (see also Remark \ref{under RH}).
\end{remark}

\begin{remark}
	The Corollary \ref{divisor functions} extends to complex values of $\alpha$, where the length of short intervals remains unchanged with respect to $\Re(\alpha)$. In this case, the constant becomes
	\begin{small}
	\[
	c(H)=\frac{\zeta(1-\Re(\alpha))^2}{\zeta(2-2\Re(\alpha))}\sum_{d=1}^{\infty}d^{2\Re(\alpha)}\left( \left\{\frac{H}{d} \right\} - \left\{\frac{H}{d} \right\} ^2\right).
	\] 
	\end{small}
\end{remark}

\begin{remark}\label{remark 6}
When $\alpha\in (-1/2, 0)$, Kiuchi and Tanigawa \cite[Corollary, eq. (5)]{KT} established that for $X^{\epsilon}\ll H\ll X^{1/2}$,
	\begin{align*}
\Var_{\sigma_{\alpha}}(X; H) \asymp H^{1+2\alpha}.
	\end{align*}
	This result exhibits greater strength within the above range of $\alpha$ when compared to our result. Particularly for $\alpha=-\frac12$, we can say that 
	\[
	\Var_{\sigma_{\alpha}}(X; H) \leq \frac{\zeta(3/2)^2}{4 \zeta(3)}\log H+ O(1).
	\]
	Proving the expected variance to be $\Theta(\log H)$ appears to pose a considerable challenge.
\end{remark}

\section{Variance for general class of multiplicative functions}\label{section 0}
 
 \subsection{General class of multiplicative functions}
The main objective of this section is to introduce general classes of multiplicative functions and analyze their local behavior. In doing so, we also uncover the local behavior of several extensively studied multiplicative functions, which are primarily discussed in Section \ref{section 1} and further explored in the Appendix (Section \ref{appendix}). 
 
Let $\mathcal{M}$ represent the set of all multiplicative functions, and $\mathcal{G}$ denote the set of all completely multiplicative functions.  Let $k\geq 1$ be an integer, and let $(\mathbf{1} *_k h)(n)=\sum_{d^k|n}h(d)$ be a generalization of the Dirichlet convolution $\mathbf{1}*h$.

Suppose $\mathcal{P}^{\text{fin}}$ represents a finite set of primes and $\beta, \eta\in \mathbb{C}$. We say an arithmetical function $r$ satisfies property $(A_{\beta})$ if there exist $\beta, \eta$ and a finite set of primes $\mathcal{P}^{\text{fin}}$ such that   
\begin{align*}
r(p)= &\left\{
\begin{array}
[c]{ll}
\pm  \beta   & \, \text{ if }  p\not \in
\mathcal{P}^{\text{fin}}, 	\\
\eta & \, \text{ if } p\in  \mathcal{P}^{\text{fin}}.
\end{array}
\right.
\end{align*}
For a parameter $\alpha\geq 0$ and $\beta\in \mathbb{C}$, we define the classes
\[
\mathcal{M}_{\alpha, \beta}:=\Big\{\mu(d)\frac{r(d)}{d^\alpha}: \, r\in \mathcal{M} \text{ and }\text{satisfies } (A_{\beta}) \Big\}
\]
and
\[
\mathcal{G}_{\alpha, \beta}:=\Big\{\frac{r(d)}{d^\alpha}:
\, r\in \mathcal{G} \text{ and }\text{satisfies } (A_{\beta}) \text{ with $|\beta|\leq 1$}\Big\}.
\]
Observe that every function $r$ occurring in the above classes satisfies
$
r(n)\ll_\varepsilon n^\varepsilon
$ for every $\epsilon>0$.
The restriction $|\beta|\le 1$ in the definition of $\mathcal{G}_{\alpha,\beta}$ is imposed precisely to ensure this bound. One can include a more general class of completely multiplicative functions for which $|\beta|>1$ but it will weaken the exponents of the short intervals described in Section \ref{exponent of short intervals}. 

Here and throughout, we assume that for $k\geq 1$ and a real number $0\leq \alpha<2$, the condition $\alpha+ k>\frac32$ holds. Now let's consider the following class of functions:
\begin{equation*}\label{general calss}
\mathcal{F}_{\alpha, \beta, k}=\bigg\{\mathbf{1} *_k h: \, h\in \mathcal{M}_{\alpha, \beta}\cup \mathcal{G}_{\alpha, \beta} \bigg\}.
\end{equation*} 
For example, $\mu_k \in \mathcal{F}_{0, 1, k}$, $\frac{\phi(n)}{n} \in \mathcal{F}_{1, 1, 1}$, and $\sigma_\alpha \in \mathcal{F}_{\alpha, 1, 1}$. Note that our main corollaries in Section~\ref{section 1} involve the case $\beta = 1$; however, to provide a more diverse family corresponding to $\mathcal{M}_{\alpha, \beta}$, we also include examples in the appendix with $\beta > 1$, for instance, the Schemmel's totient  function $S_m \in \mathcal{F}_{1, m, 1}$ for a fixed integer $m \geq 2$ (see subsection \ref{sch totient func}). 
 
\begin{remark}

%if $r\in \mathcal{G}_{\alpha, \beta}$ for any $n\in \mathbb{N}$, where $\tau_\beta(n)$ is the $n$-th coefficient of $\zeta(s)^{\beta}$ for a fixed $\beta 
%\in \mathbb{C}$.

The class $\mathcal{F}_{\alpha, \beta, k}$ can extends the family of the functions of the form $f=1*h$, and examples of this class can be found in the following subsection. To broaden the nature of $r(p)$, one can assign different values on a finite set of primes. Additionally, the above class can be further generalized by considering $r(p)=\pm \beta e^{i a(p)}$, where $p\not \in \mathcal{P}^{\text{fin}}$ and $a(n)$ is an additive function. We confine our attention to $\alpha<2$ since for $\alpha\geq 2$, the behavior of corresponding arithmetic functions becomes straightforward and is therefore omitted.
\end{remark}

\subsection{Mean values for class $\mathcal{F}_{\alpha, \beta, k}$} Given a multiplicative function $f$, it is desired to obtain an asymptotic formula for the sum $\sum_{n\leq x}f(n)$. Ideally, we aim to provide a formula accompanied by an explicit error term that depends on the values of $f(p)$ for primes $p$. 
  Let $x\geq 1$ and $f\in \mathcal{F}_{\alpha, \beta, k}$.
It is easy to see that for any $\epsilon>0$,
\begin{align}\label{mean value estimate}
\sum_{n\leq x}f(n)=\widetilde{c}_{h, k}x+O\left(x^{\max \big\{\frac{1-\alpha}{k}+\epsilon,\, 0\big\}}\right),
\end{align}
where
\begin{align}\label{constant in the short interval}
\widetilde{c}_{h, k}=\sum_{d\geq 1}\frac{h(d)}{d^k}.
\end{align}
Improvements to the error term for various arithmetical functions within the class $\mathcal{F}_{\alpha, \beta, k}$ have been achieved in multiple works, including \cite{Apostol}, \cite{JK}, \cite{JKW}, \cite{Liu3}, \cite{Moree}, \cite{Sitaram} and \cite{Wal}.
 For the class $f\in \mathcal{F}_{0, \beta, 2}$, Varbanec \cite{VAR} obtained that for any $\epsilon>0$,
\[
\sum_{x<n\leq x+H}f(n)=\widetilde{c}_{h, 2}H+O(H^{1/2}x^{\epsilon})+O(x^{\theta+\epsilon}),
\]
uniformly in $H < x$ and $\theta = 0.2204.$ 
%We observe that the error term, uniformly for $H\leq x$, exhibits an order of $O(H^{1/2+\epsilon})$. 
In this article, we establish an error term within the range of $O(H^{\frac{1-2\alpha}{2k}+\epsilon})$ with $\alpha\in [0, 1/2)$ and $O(1)$ with $\alpha\in (1/2, 2)$ for the family $\mathcal{F}_{\alpha, \beta, k}$ on average over $x\in [X, 2X]$. 
\subsubsection*{Exponents of short intervals}\label{exponent of short intervals} In order to state our results precisely, we present the exponents of short intervals based on $\alpha$ and $k$. When $k\geq 3$, we define a function $\nu:=\nu(k, \alpha)$ as follows. 
\begin{equation}\label{10/5/14:24}
\nu=\begin{cases}
\frac{-(2k-5/2+5\alpha)+\sqrt{(2k-5/2+5\alpha)^2+4k(k-2)(k+2\alpha-1)(2k-5/2+5\alpha)}}{2(k-2)(k+2\alpha-1)} & \mbox{ if } \alpha\in [0, 1/2),\\
\frac{-2\alpha+2\sqrt{\alpha^2+\alpha k(k-2)}}{k-2} & \mbox{ if } \, \alpha\in (1/2, 1).
\end{cases}
\end{equation}
Let $\alpha\in [0, 1/2)$. The exponents $e$ and $g$ are determined based on $\nu$ in the following manner:
\begin{equation}\label{10/5/14:18}
e(k, \alpha)=\begin{cases}
\frac{2(3+8\alpha)}{11(1+2\alpha)} & \mbox{ if } k=2, \vspace{2mm}\\
\frac{2(k-\nu)(k+2\alpha-5/4)}{((4-\nu)k-2\nu)(k+2\alpha-1)} & \mbox{ if } k \geq 3,
\end{cases}
\quad \mbox{and} \quad g(k)= \begin{cases}\frac23  & \mbox{ if }\, k = 2,\\
\frac{\sqrt{2k^2-4k+1}-1}{2\sqrt{2k^2-4k+1}-k} & \mbox{ if } \, k\geq 3.
\end{cases}
\end{equation}
For $\alpha\in (1/2, 1)$, the exponents are defined by
\begin{align}\label{hat e}
\widehat{e}(k, \alpha)=\begin{cases}
\min\big\{\frac{29}{113-84\alpha}, \frac{8\alpha - 1}{17-16\alpha}\big\} & \mbox{ if } \, k=1,\\
\min\big\{\frac{29}{71-42 \alpha}, \frac{2+\alpha}{4}\big\} & \mbox{ if }\, k=2,\\
\min\big\{\frac{29k}{29k+84(1-\alpha)}, \frac{\nu (k+2\alpha-5/4)}{2k(\nu- \alpha)}\big\} &\mbox{ if } \, k\geq 3,
\end{cases}
\quad \mbox{ and }\quad \widehat{g}(k, \alpha)=\begin{cases}\frac{1}{3-2\alpha} & \mbox{ if } \, k=1,\\
\frac{2+\alpha}{4} & \mbox{ if }\, k=2,\\
\frac{\nu}{2(\nu- \alpha)} &\mbox{ if } \, k\geq 3.
\end{cases}
\end{align}
Also for $\alpha\in [1, 2)$,
\begin{align}\label{hat e 1 to 2}
\widehat{e}(k, \alpha)=\widehat{g}(k, \alpha)=\begin{cases}
1 & \mbox{ if } \, k\neq 2,\\
\frac{2+\alpha}{4} & \mbox{ if }\, k=2.
\end{cases}
\end{align}
The intersection between the curves, which are determined by counting points on a binary form (refer to Section \ref{section 5}), yields these exponents. We illustrate these exponents in Figures \ref{figure 1} and \ref{figure 2} for several fixed values of $\alpha$, while $k$ varies from 1 to 100 (see the Appendix).
\subsection{Main results on variances of multiplicative functions}
By utilizing the parameters introduced in Section \ref{section 2}, we are now able to present our main results. Our objective is to derive an asymptotic formula for the variance, as expressed in the following form:
\[
\Var_f(X; H):=\frac{1}{X}\int_{X}^{2X}\bigg|\sum_{x<n\leq x+H}f(n)- \frac{H}{X}\sum_{n\leq X}f(n)\bigg|^2 dx.
\] 
For $f\in \mathcal{F}_{\alpha, \beta, k}$, using Cauchy--Schwarz inequality, $H\leq X^{1-\epsilon}$, and the estimate \eqref{mean value estimate}, it suffices to evaluate the following variance with a negligible error (comes from the error term of the global mean):
\begin{align*}\label{varinace}
\Var_f(X; H):= \frac{1}{X}\int_X^{2X}\bigg|\sum_{x<n\leq x+H}f(n)- \widetilde{c}_{h, k}H\bigg|^2 dx,
\end{align*}
where $X\geq 1$ is sufficiently large, and $\widetilde{c}_{h, k}$ is defined by \eqref{constant in the short interval}.
 
The first two theorems provide estimations of $\Var_f(X; H)$ for functions $f\in \mathcal{F}_{\alpha, \beta, k}$, where $\alpha \in [0, 1/2)$, whereas the third theorem is applicable when $\alpha\in (1/2, 2)$. 

\begin{theorem}\label{main theorem for simultaneouly k-free intergers}
	Let \(f\in \mathcal{F}_{\alpha,\beta,k}\), where
	\(0\leq \alpha<1/2\), \(\beta\in\mathbb{Z}\setminus\{0\}\), and \(k\geq 2\).
	Recall the quantities \(e(k,\alpha)\) and \(g(k)\) defined in
	\eqref{10/5/14:18}. Let \(\epsilon>0\) be sufficiently small and suppose that
	$
	2\leq H\leq X^{e(k,\alpha)-\epsilon}.
	$
	Then
	\[
	\Var_f(X;H)
	=
	c_{h,k}\,
	H^{\frac{1-2\alpha}{k}}
	P_{\beta^2-1}(\log H)
	+
	O\Bigl(
	H^{\frac{1-2\alpha}{k}-\delta_{k,\epsilon}}
	\Bigr),
	\]
	where \(c_{h,k}\) and \(P_{\beta^2-1}\) are defined in
	\eqref{constant 2} and \eqref{logarithm in the main term}, respectively.
	The constant \(\delta_{k,\epsilon}>0\), which depends on the choice of \(\epsilon\) and $k$, is described in the remark below.
	
	Under the Lindel\"of hypothesis, the above asymptotic formula remains valid in the wider range
	$2\leq H\leq X^{g(k)-\epsilon}.$
	
\end{theorem}

\begin{remark}
	Define
	\[
	\epsilon_{\alpha,\beta,k}
	:=
	\min\left\{
	\frac{3(1-2\alpha)}
	{2(\lceil \beta^2/12\rceil+4)},
	\,
	\frac{9(1-2\alpha)}
	{3+k\beta^2}
	\right\}.
	\]
	Then one may choose 
$
	\epsilon<\epsilon_{\alpha,\beta,k},
$
in the Theorem \ref{main theorem for simultaneouly k-free intergers} so that
	\[
	\delta_{2,\epsilon}
	=
	\frac{\epsilon}{8},
	\qquad
	\delta_{k,\epsilon}
	=
	\frac{\epsilon}{3(k+1)},
	\quad k\geq 3.
	\]
\end{remark}
Note that Theorem  \ref{main theorem for simultaneouly k-free intergers} holds true for any non-zero integer $\beta$. We will now extend the theorem to include any complex $\beta$, accompanied by a weaker error term, as described below.
\begin{theorem}\label{main theorem complex case}
Let $f\in \mathcal{F}_{\alpha, \beta, k}$, where $0\leq \alpha< 1/2, \beta\in \mathbb{C}\setminus \mathbb{Z}$, and $k\geq 2$. Suppose that $e(k, \alpha)$ and  $g(k)$ are defined by \eqref{10/5/14:18}. For $2\leq H\leq X^{e(k, \alpha)-\epsilon}$ and $N\geq  \lceil |\beta|^2\rceil$, we have
\[
\Var_f(X; H)=2H^{\frac{1-2\alpha}{k}}(\log H)^{|\beta|^2-1}\bigg\{\sum_{0\leq j\leq N}\frac{k^{j-|\beta|^2}\lambda_j(|\beta|^2)}{\log^j H}+O\left(\frac{1}{\log^{N+1} H}\right)\bigg\},
\]
where $\lambda_j(|\beta|^2)$ is defined by \eqref{definition of lambda}.
Assuming the Lindel\"{o}f hypothesis,
 the above estimate holds in the wider range  $H\leq X^{g(k)-\epsilon}$.
\end{theorem}

 \begin{remark}\label{remark 3}
One can consider the case $k=1$ in the above theorem. In this scenario, it is necessary to take into account the condition $\alpha>\frac{13}{84}$, which is derived by using Bourgain's bound to control the length of a specific Dirichlet polynomial (see \eqref{27/06/23}). Although the range of $H$ does not surpass the $\frac12$-barrier, and thus we omit this specific case.  
	\end{remark}

\begin{theorem}\label{main theorem for lower growth rate}
 Let $\frac{1}{2}<\alpha <2$, $\beta\in \mathbb{Z}\setminus \{0\}$ and $k\geq 1$.
  Let $\epsilon>0$ be given and $2\leq H\leq X^{\widehat{e}(k, \alpha)-\epsilon}$, where $\widehat{e}(k, \alpha)$ be defined by \eqref{hat e}-\eqref{hat e 1 to 2}. Then
\begin{align*}
\Var_f(X; H)=c_{h, k}(H)+O(H^{-\epsilon_{\alpha, k}}),
\end{align*}
where $c_{h, k}(H)$ is as in \eqref{C_h,k} and  $\epsilon_{\alpha, k}>0$ depends on $\epsilon, \alpha$ and $k$.

Assuming the Lindel\"{o}f hypothesis, the above estimate holds in a wider range  $2 \leq  H\leq X^{\widehat{g}(k, \alpha)-\epsilon}$, where $\widehat{g}(k, \alpha)$ is defined by \eqref{hat e}. 
 	
\end{theorem}

\begin{remark}\label{remark 10}
For the case of $\alpha=\frac12$, we are unable to derive the asymptotic formula for the variance directly. However, by following the proof of Theorem \ref{main theorem for lower growth rate}, we can establish
 \[
\Var_f(X; H)\ll (\log H)^{\beta^2},
\]
where the implied constant can be explicitly determined from the corresponding family (likewise the Remark \ref{remark 6}). Nevertheless, we encountered a technical obstacle in employing the Fourier analytic method to approach Theorem \ref{main theorem for simultaneouly k-free intergers}. This obstacle arises from the fact that certain equations, such as \eqref{neq2} and \eqref{bound of hatw}, do not hold when $\alpha=\frac12$.  
\end{remark}

\begin{remark}\label{under RH}
	While the Lindel\"{o}f hypothesis falls short of achieving our desired range $H\leq X^{1-\epsilon}$, the Riemann hypothesis furnishes us with a weaker estimate. Under the Riemann hypothesis, we can show that for $H\leq X^{1-\epsilon}$, the following holds for the restriction on $\beta$:
	\begin{small}
	\begin{align*}
	\Var_f(X; H)\ll \begin{cases}
	H^{\frac{1-2\alpha}{k}+\epsilon} & \mbox{ if } 0\leq \alpha< \frac12,\\
	H^{\epsilon} &  \mbox{ if } \frac12< \alpha< 1.	
	\end{cases}
	\end{align*}
	\end{small}
 In particular, we need to consider $\beta>0$ whenever $h\in \mathcal{M}_{\alpha, \beta}$ and $\beta<0$ if $h\in \mathcal{G}_{\alpha, \beta}$. To achieve this, we can modify the proof of Proposition \ref{19.04.22.6:26} based on the chosen value of $\alpha$. By doing so, we can effectively control the length of the Dirichlet polynomial while imposing the above restrictions on $\beta$ under the Riemann hypothesis. Moreover, in this scenario, we can only establish an upper bound for the variance, with an additional $\epsilon$ factor relative to its original asymptotic behavior.
	\end{remark}
%	\begin{remark}
%	In our class, the function $|r(p)|$ is eventually constant; however, the above results can be extended to situations where $|r(p)|$ does not eventually constant.
%	\end{remark}

\subsection{Essence of the paper} 
In Section 4, we consider a quantity $\Var_f^{\mathcal{D}}(X; H)$, which is a ``discrete variance". We refer to the parameters $\alpha$ and $k$ as the ``growth" and ``order", respectively, for a given arithmetic function $f$. These parameters play a crucial role in computing the asymptotic formulas for both $\Var_f(X; H)$ and $\Var_f^{\mathcal{D}}(X; H)$.
Conceptually, treating the discrete and continuous variances separately serves the purpose: while the former detects small values of $H$ without surpassing the half barrier, the latter overcomes this limitation in Fourier space. However, the continuous variance fails to detect small values of $H$, particularly when $H\ll X^{\epsilon}$.

The computation of $\Var_f^{\mathcal{D}}(X; H)$ involves standard techniques such as the Chinese remainder theorem and complex analytical method, while $\Var_f(X; H)$ involves a Fourier analytical method. Specifically, we follow the approach of Hall \cite{Hall} for $\Var_f^{\mathcal{D}}(X; H)$ and the method proposed by Gorodetsky et al. \cite{GMRR} for $\Var_f(X; H)$.

 The main difficulty arises from the additional influence of the growth and order factors. Specifically, Gorodetsky et al. \cite{GMRR} focused on estimating the ``second order" ($k=2$) and ``zero growth" ($\alpha =0$) scenarios for square-free integers. Additionally, for higher orders ($k\geq 3$), we employ a method based on counting rational points on specific binary forms, as described in Section 6, instead of using Dirichlet approximation as in \cite[subsection 3.4]{GMRR}. Generalizing the Dirichlet approximation method for higher order appears to be quite challenging. Notably, for $k=1$, we utilize point counting on a determinant surface.

Furthermore, the continuous variance system is calculated through two distinct approaches: one involves partitioning the convolution $f(n)=\sum_{d^k\mid n}h(d)$ into smaller divisors, which yields the main term for the asymptotic analysis using Fourier analytical techniques and counting rational points in binary form; the other approach focuses on computing the larger divisors by effectively managing the length of the corresponding Dirichlet polynomial, utilizing Huxley's large value theorem, Bourgain's sub-convexity bound (effective than the Weyl bound), and Montgomery's mean value theorem. We mainly follow the approach of \cite{GMRR} where the additional tools to count rational points on a binary form and Bourgain's sub-convexity bound has been applied to handle higher order arithmetic functions described as above.

%One of our paper's main obstacles comes from considering higher order and growth for a certain arithmetic function. The new techniques to overcome these obstacles are mainly described in Section \ref{6.2}.
It is noted that the variance, denoted as $\Var_f(X; H)$, undergoes a shift in order of magnitude, specifically $H^{\frac{1-2\alpha}{2k}}$, as the parameter $\alpha$ changes. When $\alpha=\frac12$, which is the critical point, no asymptotic formula has been derived yet. However, we do have the inequality $\Var_f(X; H)\ll (\log H)^{|\beta|^2}$. Subsequently, as $\alpha$ crosses the critical point and approaches $1$, the variance remains constant and only depends on the length of short intervals.

\subsection{Organization of the paper} This article is structured as follows: Section \ref{section 2} provides definitions of basic notations and parameters used in the main theorems. Section \ref{section 3} presents key propositions that support the proofs of the main theorems stated in Sections \ref{section 1} and \ref{section 0}. All main theorems and corollaries are proven in Section \ref{section 4}, building upon the propositions of Section \ref{section 3}. Section \ref{section 5} establishes the validity of Propositions \ref{below z-barrier} and \ref{proposition for lower growth family}, which address small divisors. In Section \ref{section 6}, we prove Proposition \ref{19.04.22.6:26} concerning large divisors. Finally, Section \ref{section 7} concludes by presenting the proof of Propositions ~\ref{prop5} and ~\ref{prop6}. Furthermore, we formulate several other interesting applications and figures related to exponents of our main results in Section \ref{appendix}.

\section*{Acknowledgments} The authors express their sincere gratitude to Brad Rodgers and Ofir Gorodetsky for their diligent review of the initial version of this manuscript, as well as for their valuable comments and insightful suggestions. This research was initiated during the first author's NBHM Fellowship at ISI Kolkata, India, which was funded by DAE, while the second author was a visiting scientist at that time. The work was completed during the first author's ERCIM ``Alain Bensoussan" Fellowship at NTNU, Trondheim, Norway, and the second author's post-doctoral fellowship at the Institute of Mathematical Sciences, India. Additionally, Darbar acknowledges partial support from the Research Council of Norway grant 275113, and Das thanks DST - Government of India for the support under the
DST-INSPIRE Faculty Scheme with Faculty Reg. No. IFA21--MA 168. 

\section{Parameters used in main theorems}\label{section 2}
 
In this paper, we use the following notations. For a non-negative function $g(n)$, we write $f(n)=O(g(n))$ (or $\ll g(n)$) if $|f(n)|\leq C g(n)$ holds for all $n>N$, where $C>0$ is a constant.
%If $C\leq 1$, we refer to it as $f(n)=\Theta (g(n))$.

We say that $f(n)=o(g(n))$ if the ratio $\frac{f(n)}{g(n)}\to 0$ as $n\to \infty$. Moreover, we write $f(n)\asymp g(n)$ or $f(n)=\Theta(g(n))$ if there exist positive constants $C_1$ and $C_2$ such that $C_1 g(n)\leq |f(n)|\leq C_2 g(n)$. Additionally, we write $n\sim N$ to denote $N\leq n\leq 2N$.

\noindent Furthermore, we introduce several functions that play a crucial role in stating our results.

The function $\chi(s)$ is defined on $\mathbb{C}\setminus \{1\}$, and it is given by the expression 
\begin{equation}\label{10/5/00:43}
\chi(s)= 2^s\pi^{s-1}\Gamma(1-s)\sin(\pi s/2) .
\end{equation}
This function arises in the functional equation of the Riemann zeta function,  \[\zeta(s)=\chi(s)\zeta(1-s).\]

\subsection*{In the range $\alpha\in [0, 1/2)$}
Let us consider $\beta\in \mathbb{C}\setminus \{0\}$ and $k\geq 2$. For $s\in \mathbb{C}$, we introduce the following definitions, which are used to compute residues of certain line integrals:
\begin{small}
\begin{align}\label{B(s)}
B(s):=\prod_{p} \left( 1+\frac{2\Re (h(p))}{p^{k}} + \frac{|h(p)|^2}{p^{k+ ks}} \right)\left(1-\frac{1}{p^{k+ks+2\alpha}}\right)^{|\beta|^2}
\end{align}
\end{small}
and 
\begin{small}
\begin{align}\label{D(s)}
D(s):=\prod_{p}\left(1-\frac{|h(p)|^2}{p^{k+ks}}\right)^{-1}\left(1-\frac{1}{p^{k+ks+2\alpha}}\right)^{|\beta|^2}
\left(1-\frac{|h(p)|^2}{p^{2k}}\right)\left(1-\frac{2\Re (h(p))}{p^k}+\frac{|h(p)|^2}{p^{2k}}\right)^{-1}.
\end{align}
\end{small}
Recalling the function $\chi(s)$ from \eqref{10/5/00:43}, we define $c_{h, k}$ for $\beta\in \mathbb{Z}$ as follows:
\begin{small}
\begin{align}\label{constant 2}
c_{h, k}=&\left\{
 	\begin{array}
 	[c]{ll}
 	 -\frac{1}{2\pi^2(\beta^2-1)!}B\left(\frac{1-2\alpha-k}{k}\right)\chi(\frac{k+1-2\alpha}{k})\zeta\left(2-\frac{1-2\alpha}{k}\right) \,& \text{ if   } h\in \mathcal{M}_{\alpha, \beta}, 	\\
-\frac{1}{2\pi^2(\beta^2-1)!}D\left(\frac{1-2\alpha-k}{k}\right)\chi(\frac{k+1-2\alpha}{k})\zeta\left(2-\frac{1-2\alpha}{k}\right) & \, \text {if } h\in \mathcal{G}_{\alpha, \beta}
.
 	\end{array}
 	\right.
\end{align} 
\end{small}
We express $P_{t}(x)$ as follows:
\begin{small}
\begin{align}\label{logarithm in the main term}
P_{t}(x)=&\left\{
 	\begin{array}
 	[c]{ll}
 	\sum_{\substack{r+l+m+n=t\\ r,l,m ,n\geq 0}}\Big(\frac{B^{(r)}}{B}\cdot \frac{G_{k}^{(l)}}{G_{k}}\cdot\frac{\chi^{(m)}}{\chi}\Big)\left(\frac{1-2\alpha-k}{2\pi ik}\right) x^{n} \,& \text{ if   } h\in \mathcal{M}_{\alpha, \beta},  	\\
\sum_{\substack{r+ l +m+ n=t\\ r, l, m, n \geq 0}}\Big(\frac{D^{(r)}}{D}\cdot \frac{G_{k}^{(l)}}{G_{k}}\cdot\frac{\chi^{(m)}}{\chi}\Big)\left(\frac{1-2\alpha-k}{2\pi ik}\right) x^{n} & \, \text {if } h\in \mathcal{G}_{\alpha, \beta}.
 	\end{array}
 	\right.
\end{align}
\end{small}
Here, $P_{t}(x)$ is a monic polynomial in $x$ of degree $t$, the $n$-th derivative of $f$ is denoted by $f^{(n)}$, and $(f\cdot g \cdot h)(y):=f(y)g(y)h(y)$. The function $G_k(s)$ is defined as
\begin{small}
\begin{align}\label{G(s)}
 G_k(s):=\zeta(1-s)\zeta^{\beta^2}(k+ks+2\alpha)\left(s-\frac{1-2\alpha-k}{k}\right)^{\beta^2}.
\end{align}
 \end{small}
Furthermore, for complex $\beta$, we introduce the function $\lambda_{j, \alpha}(|\beta|^2)$ given by
\begin{small}
\begin{align}\label{definition of lambda}
\lambda_{j, \alpha}(|\beta|^2):=\frac{1}{\Gamma(|\beta|^2-j)}\sum_{h+i=j}\frac{1}{i! h!}L_{\alpha}^{(h)}(1)\gamma_i(|\beta|^2),
\end{align}
\end{small}
where the coefficients $\gamma_i$ are obtained from the Taylor series expansion of
\begin{small} 
\begin{align}\label{definition of gamma}
\frac{1}{z}((z-1)\zeta(z))^{|\beta|^2}=\sum_{j\geq 0}\frac{1}{j!}\gamma_j(|\beta|^2)(z-1)^j,
\end{align}
\end{small}
and $L_{\alpha}(z)$ is defined by
\begin{small}
\begin{align*}
L_{\alpha}(z):=
 	\begin{cases}
 	 z\, \zeta\left(2-\frac{z-2\alpha}{k}\right)B\left(-1+\frac{z-2\alpha}{k}\right)\chi \big(\frac{-1+(z-2\alpha)/k}{2\pi i}\big) \,& \text{ if   } h\in \mathcal{M}_{\alpha, \beta},\\
 z\, \zeta\left(2-\frac{z-2\alpha}{k}\right)D\left(-1+\frac{z-2\alpha}{k}\right)\chi \big(\frac{-1+(z-2\alpha)/k}{2\pi i}\big) & \, \text {if } h\in \mathcal{G}_{\alpha, \beta}.
 	\end{cases}
 \end{align*}
 \end{small}
\subsection*{The range $\alpha\in (1/2, 2)$}
We introduce a constant, denoted by $c_{h,k}(H)$, which depends on the parameter $H$. It is defined as follows:
\begin{small}
\begin{align}\label{C_h,k}
c_{h,k}(H):= 
\begin{cases}
\sum_{d=1}^{\infty}h^2(d)\prod_{p\nmid d}\Big(1 + \frac{2h(p)}{p^k}\Big)\Big( \left\{ \frac{H}{d^k} \right\} - \left\{\frac{H}{d^k} \right\} ^2\Big) & \mbox{ if } h\in \mathcal{M}_{\alpha, \beta}, \\
\Big(\prod_p\frac{p^k+h(p)}{p^{k}-h(p)}\Big)\sum_{d=1}^{\infty}h^2(d)\Big( \left\{ \frac{H}{d^k} \right\} - \left\{\frac{H}{d^k} \right\} ^2\Big) & \mbox{ if } h\in \mathcal{G}_{\alpha, \beta}.
\end{cases}
\end{align}
\end{small}

\section{Preparation for proving Theorem \ref{main theorem for simultaneouly k-free intergers} and Theorem~\ref{main theorem for lower growth rate}}\label{section 3}

In order to establish these theorems, we rely on the main propositions presented in Sections \ref{section 5}, \ref{section 6}, and \ref{section 7}.
The first two propositions provide us with an asymptotic formula for the variance of a restricted sum up to a certain parameter $z$. Meanwhile, the third proposition provides an upper-bound estimate for the same quantity beyond $z$. For the purpose of our analysis, we denote these constrained variances as $\mathcal{J}_{h, z}(X; H)$ and $\mathcal{K}_{h, z}(X; H)$, respectively. In other words, for sufficiently large $X$,
\begin{small}
\begin{align}\label{J_k}
\mathcal{J}_{h, z}(X; H):= \frac{1}{X}\int_X^{2X} \Bigg|\sum_{\substack{x<nd^k\leq x+H\\ d^k\leq z}}h(d)-H\sum_{d^k\leq z}\frac{h(d)}{d^k}\Bigg|^2 dx,
\end{align}
\end{small}
\begin{small}
\begin{align}\label{K_k}
\mathcal{K}_{h, z}(X; H):= \frac{1}{X}\int_X^{2X} \Bigg|\sum_{\substack{x<nd^k\leq x+H\\ d^k> z}}h(d)-H\sum_{d^k> z}\frac{h(d)}{d^k}\Bigg|^2 dx.
\end{align}
\end{small}

\begin{proposition}\label{below z-barrier}
Let $f\in \mathcal{F}_{\alpha,\beta, k}$ with $\beta\in \mathbb{Z}\setminus \{0\}$ and $0\leq \alpha<\frac{1}{2}$. Suppose that $0<\epsilon <  \frac{1-2\alpha}{2(\lceil\beta^2/12\rceil+4)}$ is given.
\begin{enumerate}[label=(\alph*)]
	\item
 Assume $k=2$ and $X^{\epsilon}\leq H\leq X^{\frac{2+\alpha}{3+2\alpha}-\epsilon}$.
For a parameter $z$ with
\[
H^{1+\epsilon}\leq z \leq \min\Big\{X^{\frac{1}{1-\alpha}}H^{-\frac{1+2\alpha}{2-2\alpha}-\epsilon}, H^{\frac{1-2\alpha}{2-2\alpha}- \epsilon}X^{\frac{1}{2-2\alpha}}\Big\},
\]
we obtain
\begin{align*}
\mathcal{J}_{h,z}(X; H)= 
 	c_{h, 2}H^{\frac{1-2\alpha}{2}}P_{\beta^2-1}(\log H)
 +O\left(H^{\frac{1-2\alpha}{2}-\frac{\epsilon}{5}}\right). 	
\end{align*}
\item
 For $k\geq 3$ and $H^{1+\epsilon}\leq z\leq  \min \Big\{X^{\frac{2}{\nu(2-\alpha)}}H^{\frac{1-2\alpha}{\nu(2-\alpha)}-\epsilon}, X^{\frac{\nu}{2(\nu-\alpha)}}H^{\frac{1-2\alpha}{2(\nu-\alpha)}-\epsilon}, X^{\frac{2(k-\nu)}{(4-\nu)k-2\nu}-\epsilon}\Big\}$,
\begin{align*}
\mathcal{J}_{h, z}(X; H)=
 	c_{h, k}H^{\frac{1-2\alpha}{k}}P_{\beta^2-1}(\log H)
 + O\left(H^{\frac{1-2\alpha}{k}-\frac{\epsilon}{k}} \right),
\end{align*}
where $\nu, c_{h, k}$ and $P_{\beta^2-1}(\log H)$ are defined as in \eqref{10/5/14:24}, \eqref{constant 2} and \eqref{logarithm in the main term} respectively.
\end{enumerate}  
\end{proposition}

\begin{proposition}\label{proposition for lower growth family}
Consider $f\in \mathcal{F}_{\alpha,\beta, k}$ with $\beta\in \mathbb{Z}\setminus \{0\}$ and $k\geq 1$. Let $\epsilon>0$ be given, and suppose there exists a positive $\epsilon_{\alpha, k}>0$. Assume that $\frac12<\alpha< 1$. Then
\[
\mathcal{J}_{h,z}(X; H)=
c_{h,k}(H)+ O\left(H^{-\epsilon_{\alpha, k}}\right),
\]
provided that
\begin{align*}
z\leq 
\begin{cases}
X^{\frac{1}{2(2-\alpha)}}H^{\frac{1}{2(2-\alpha)}-\epsilon} & \mbox{ if }  k=1,\\
\min\big\{X^{\frac{1}{1-\alpha}}H^{-\frac{1}{1-\alpha}-\epsilon}, X^{\frac{1}{2-2\alpha}}H^{-\epsilon}\big\} & \mbox{ if } H\leq X^{\frac{2+\alpha}{4}-\epsilon} \text{ and }  k=2,\\
\min \big\{X^{\frac{2}{\nu(2-\alpha)}-\epsilon}, X^{\frac{\nu}{2(\nu-\alpha)}-\epsilon}, X^{\frac{2(k-\nu)}{(4-\nu)k-2\nu}-\epsilon}\big\} & \mbox{ if } k\geq 3.
\end{cases}
\end{align*}
Suppose $1\leq \alpha< 2$. In this scenario, the above estimate holds whenever 
	\begin{align*}
	z\leq 
	\begin{cases}
	X^{\frac{1}{2(2-\alpha)}}H^{\frac{1}{2(2-\alpha)}-\epsilon} & \mbox{ if }  k=1,\\
	X^{1-\epsilon} & \mbox{ if } H\leq X^{\frac{2+\alpha}{4}-\epsilon} \text{ and }  k=2,\\
	X^{1-\epsilon} & \mbox{ if } k\geq 3.
	\end{cases}
	\end{align*}
\end{proposition}
The next proposition provides an upper bound for $\mathcal{K}_{h,z}(X; H)$ for various ranges of $\alpha, \beta$ and $k$.
\begin{proposition}\label{19.04.22.6:26}
Let $f\in \mathcal{F}_{\alpha, \beta, k}$ with $\beta\in \mathbb{C}\setminus \{0\}$ and $\epsilon>0$ be given.
\begin{enumerate}[label=(\alph*)]
	\item  Assume that $0\leq \alpha<\frac12$, $k\geq 2$ and $X^\epsilon\leq H\leq X^{\frac{29k}{29k+42}-\epsilon}$. If $z\geq H^{\frac{k+2\alpha-1}{k+2\alpha-5/4}+\epsilon}$  then we have,
\begin{align*}
\mathcal{K}_{h,z}(X; H)\ll H^{\frac{1-2\alpha}{k}-\frac{\epsilon}{3}}.
\end{align*}
Under the Lindel\"{o}f hypothesis, the claim holds for $X^\epsilon\leq H\leq X^{\frac{k}{k+1}-\epsilon}$ and $z\geq H^{1+\epsilon}$.

\item
Assume that $\frac12< \alpha< 1$, $k\geq 1$, and $X^\epsilon\leq H\leq X^{\frac{29k}{29k+84(1-\alpha)}-\epsilon}$. If $z\geq H^{\max\big\{1, \frac{k}{k+2\alpha-5/4}\big\}+\epsilon}$ then we get
\[
\mathcal{K}_{h,z}(X; H)\ll H^{-\frac{\epsilon}{3}}.
\]
Under the Lindel\"{o}f hypothesis, we have the same estimate, whenever $X^\epsilon\leq H\leq X^{\frac{k}{k-2\alpha+2}-\epsilon}$ and $z\geq H^{1+\epsilon}$.

\item Assume that $1\leq \alpha< 2$, $k\geq 1$, and $H\geq X^\epsilon$. If $z\geq H^{1+\epsilon}$ then we get
\[
\mathcal{K}_{h,z}(X; H)\ll H^{-\frac{\epsilon}{3}}.
\]
\end{enumerate}
\end{proposition}

\begin{remark}
Note that when $\frac12< \alpha< 1$, the exponents of short intervals in Theorem \ref{main theorem for lower growth rate} are influenced by the range $H\leq X^{\frac{29k}{29k+84(1-\alpha)}-\epsilon}$ from Proposition \ref{19.04.22.6:26}. In this context, we used the best known Bourgain's bound instead of the conventional sub-convexity bound for the Riemann zeta function, resulting in improved exponents.
\end{remark}	
Let us now establish the discrete variance of $f\in \mathcal{F}_{\alpha, \beta, k}$ with $\beta\in \mathbb{Z}\setminus \{0\}$ in short intervals, assuming that $X$ is sufficiently large. The discrete variance is defined as follows:
\begin{align}\label{discrete variance2}
\Var_f^{\mathcal{D}}(X; H):=\frac{1}{X}\sum_{n\leq X}\left(\sum_{j=1}^H f(n+j) -\widetilde{c}_{h, k}H \right)^2.
\end{align}
The exponent of short intervals is defined as follows: 
\begin{align}\label{exponent discrete case}
\theta(k, \alpha):=\begin{cases}
\frac{k(k+\alpha-1)}{(k-\alpha+1)(2k+2\alpha-1)} & \mbox{ if } \quad 0\leq \alpha <\frac{1}{2},\\
	\frac{k+\alpha-1}{2(k-\alpha+1)} & \mbox{ if } \quad \frac12< \alpha <2.	
		\end{cases}
		\end{align}
		
\begin{proposition}\label{prop5}
Assume that $0<\epsilon<\frac{3(1-2\alpha)}{k(k\beta^2+3)}$ with $\alpha\in [0, 1/2)$ and $k\geq 2$.  Let $\theta(k, \alpha)$ be as in \eqref{exponent discrete case} such that $2\leq H \leq X^{\theta(k, \alpha)-\epsilon}$. Then we have
\begin{align*}
\Var_f^{\mathcal{D}}(X; H)=
 	c_{h, k}H^{\frac{1-2\alpha}{k}}P_{\beta^2-1}(\log H)
   +O\left(H^{\frac{1-2\alpha}{k}-\epsilon}\right),
\end{align*}
where $c_{h, k}$ and $P_{\beta^2-1}(\log H)$  are defined by  \eqref{constant 2} and \eqref{logarithm in the main term} respectively.
\end{proposition}

\begin{remark}
	Observe that Proposition \ref{prop5} establishes the main result of Avdeeva \cite{Av1} and Proposition 1.7 of \cite{GMR} in specific instances involving $B$-free integers, in particular $k$-free integers.
\end{remark}

\begin{proposition}\label{prop6}
Let $\epsilon >0$ be given, $1/2<\alpha<2$, and $\theta(k, \alpha)$ be as in \eqref{exponent discrete case}. For $k\geq 1$ and $2\leq H \leq  X^{\theta(k, \alpha)-\epsilon}$, 
\begin{align*}
&\Var_f^{\mathcal{D}}(X; H)= c_{h, k}(H)+O(H^{-\epsilon}).
\end{align*}
\end{proposition}

\begin{remark}
Recalling $e(k, \alpha)$ from \eqref{10/5/14:24}, we observe that in the case of continuous variance, the range $X^{\epsilon}\leq H\leq X^{e( k, \alpha)-\epsilon}$ holds. However, in the discrete case, the range narrows to $2\leq H\leq X^{\theta(k, \alpha)-\epsilon}$. Notably, $\theta(k, \alpha)$ remains bounded by $\frac12$, and overcoming this $\frac12$-barrier through a discrete process appears to be highly challenging. Additionally, $e(k, \alpha)$ significantly surpasses $\theta(k, \alpha)$ as depicted in Figure \ref{continuous and discrete exponents} (refer to the Appendix).  
\end{remark}

\section{Proof of Proposition \ref{below z-barrier} and Proposition \ref{proposition for lower growth family}}\label{section 5}
We start by introducing several key lemmas that are crucial for proving the propositions. The lemmas presented in the initial subsections involve applications of complex and Fourier analysis techniques as outlined in \cite{GMRR}. In the subsequent subsection, we focus on estimating a specific double sum by employing a method that counts rational points on a particular binary form. By using the Fourier series of Bernoulli's polynomial, we derive the subsequent standardized formula.
\begin{lemma}\label{lem1}
 For a real number $x$, we have 
 \begin{align*}
 \sum_{n =1}^{\infty}\frac{\cos{(2\pi n x)}}{n^2} = \pi^2\left(\{x\}^2 - \{x\}\right) + \zeta(2),
 \end{align*}
 where $\{x\}$ denotes the fractional part of $x$.
 \end{lemma}

\subsection{Application of Fourier analysis and the Cauchy residue theorem}

To establish our lemmas, we shall examine the integral 
 \[   
 \int_0^\infty  x^{\frac{k+2\alpha-1}{k}}S(x)^2(\log x)^j dx , \quad j=0, 1, \ldots, \beta^2-1,
\]
where $S(x)= \frac{\sin \pi x}{\pi x}$ is defined to be continuous at $x=0$, and $k \geq 2$. Referring to~\cite[Formula 3.823]{GR14}, we observe that for $-2< \Re({z}) < 0$,
\begin{align*}
\int_0^\infty x^{z+1}S(x)^2 dx = -\frac{1}{2\pi^2}\frac{\Gamma(z)\cos(\frac{\pi z}{2})}{(2\pi)^z}= -\frac{\chi(1-z)}{4\pi^2}, 
\end{align*}
where $\chi$ is defined by \eqref{10/5/00:43}.
To further analyze this expression, we apply differentiation under the integral sign at $z=\frac{2\alpha-1}{k}$ to obtain the following lemma.
\begin{lemma}\label{nice int}
	Let $S(x)$ and $\chi$ be defined as above. For $j=0, 1, \ldots, \beta^2-1$, we have
\begin{align}\label{neq2}
\int_0^\infty  x^{\frac{k+2\alpha-1}{k}}S(x)^2(\log x)^j dx = \frac{(-1)^{j+1}}{4\pi^2}\chi^{(j)}\Big(1+\frac{1-2\alpha}{k}\Big).
\end{align}
\end{lemma}
\begin{lemma}\label{lemma of main term}
Let $\beta \in \mathbb{Z}\setminus \{0\}$ and $0\leq \alpha <1/2$. Let $0<\epsilon <  \frac{1-2\alpha}{2(\lceil\beta^2/12\rceil+4)}$ be given and $h\in \mathcal{M}_{\alpha, \beta}\cup \mathcal{G}_{\alpha, \beta}$.
Then for $z\geq H^{1+\epsilon}$, we have 
\begin{align*}
I_{S, z}:=2H^2\sum_{d_1^k, d_2^k\leq z}\frac{h(d_1)h(d_2)}{d_1^kd_2^k}\sum_{\lambda\geq 1}S\left(\frac{H\lambda}{(d_1^k, d_2^k)}\right)^2=
c_{h, k}H^{\frac{1-2\alpha}{k}}P_{\beta^2-1}(\log H)
+O\left(H^{\frac{1-2\alpha}{k}-\theta_{k, \epsilon}}\right), 
\end{align*}
where $\theta_{2, \epsilon}=\frac{\epsilon}{5}$ and $\theta_{k, \epsilon}=\frac{\epsilon}{k}$ for $k\geq 3$,  $c_{h, k}$ and $P_{\beta^2-1}(\log H)$  are defined by \eqref{constant 2} and \eqref{logarithm in the main term}  respectively.

\begin{proof}[Proof of Lemma \ref{lemma of main term} \textbf{\it $($Approximate by a smooth function$)$}]
 Let us consider a smooth function $W: \mathbb{R}\to \mathbb{R}$ that serves as an approximation of $S(y)$. We require $W$ to satisfy the following conditions for all $\ell, m \in \{0, 1, \ldots, \lceil\beta^2/12\rceil+2\}$,

\begin{align}\label{bounds of smoth functions}
|W^{(m)}(y)|\leq K_0\frac{H^{\frac{\ell}{2k}\epsilon}}{1+|y|^\ell}, \quad \forall y\in \mathbb{R}.
\end{align}
We begin by working with the function $W$ in place of $S$ and want to compute $I_{W, z}$. We rewrite 
\begin{align*}
I_{W, z}=I_W+\mathcal{E}_{W, z},
\end{align*}
where
\[
I_W:=2H^2\sum_{d_1, d_2\geq 1}\frac{h(d_1)h(d_2)}{d_1^kd_2^k}\sum_{\lambda\geq 1}W\left(\frac{H\lambda}{(d_1^k, d_2^k)}\right)^2, 
\]
\[
\mathcal{E}_{W, z}:=2H^2\sum_{\substack{\max\{d_1, d_2\}>z^{1/k}}}\frac{h(d_1)h(d_2)}{d_1^kd_2^k}\sum_{\lambda\geq 1}W\left(\frac{H\lambda}{(d_1^k, d_2^k)}\right)^2.
\]
We now aim to bound the quantity $\mathcal{E}_{W, z}$. To do this, we use the bound for $W(y)$ from \eqref{bounds of smoth functions} in order to obtain 
\begin{align*}
\sum_{\lambda\geq 1}W(\lambda/u)^2\ll \sum_{\lambda\leq uH^{\frac{\epsilon}{2k}}}1+ \sum_{\lambda>u H^{\frac{\epsilon}{2k}}}\frac{H^{\frac{\epsilon}{k}}u^2}{\lambda^2}\ll
&\left\{
 	\begin{array}
 	[c]{ll}
 	u^2H^{\frac{\epsilon}{k}} & \, \text{ if   }  0<u\leq H^{-\frac{\epsilon}{2k}}, 	\\
uH^{\frac{\epsilon}{2k}} & \, \text{ if } u> H^{-\frac{\epsilon}{2k}}.
 	\end{array}
 	\right.
\end{align*}

Consequently, this allows us to obtain
\begin{align*}
\mathcal{E}_{W, z}&\ll H^{\epsilon/k}\sum_{\substack{d_1>z^{1/k}, d_2\geq 1\\ (d_1, d_2)^k\leq H^{1-\frac{\epsilon}{2k}}}}\frac{|h(d_1)h(d_2)|(d_1, d_2)^{2k}}{d_1^kd_2^{k}}+H^{1+\frac{\epsilon}{2k}}\sum_{\substack{d_1>z^{1/k}, d_2\geq 1\\ (d_1, d_2)^k>H^{1-\frac{\epsilon}{2k}}}}\frac{|h(d_1)h(d_2)|(d_1, d_2)^k}{d_1^kd_2^{k}}\\
&:=\mathcal{E}_{W, z}^{'} +\mathcal{E}_{W, z}^{''}.
\end{align*}

We first compute $\mathcal{E}_{W, z}^{''}$. Observe that $h(d)\ll_{\epsilon_1} \frac{1}{d^{\alpha-\epsilon_1}}$ for any sufficiently small $\epsilon_1>0$. 
%\begin{align*}
%\mathcal{E}_{W, z}^{''}\ll H^{1+\frac{\epsilon}{2k}}\sum_{\substack{d_1>z^{1/k}, d_2\geq 1\\ (d_1, d_2)^k>H^{1-\frac{\epsilon}{2k}}}}\frac{(d_1, d_2)^k}{(d_1d_2)^{k+\alpha-\epsilon_1}}.
%\end{align*} 
Let us denote $(d_1, d_2)=d_0$, then $d_i=b_i d_0$ with $i=1,2$. For $k\geq 2$,
\begin{align*}
& \mathcal{E}_{W, z}^{''} \ll H^{1+\frac{\epsilon}{2k}} \sum_{d_0^k\geq H^{1-\frac{\epsilon}{2k}}}\frac{1}{d_0^{k+2\alpha-2\epsilon_1}}\sum_{\substack{b_1\geq H^{(1+\epsilon)/k}/d_0\\ b_2\geq 1}}\frac{1}{(b_1b_2)^{k+\alpha-\epsilon_1}}\\
&\ll H^{1+\frac{\epsilon}{2k}}\sum_{H^{1/k-\epsilon/2k^2}\leq d_0\leq H^{(1+\epsilon)/k}}\frac{1}{d_0^{k+2\alpha-2\epsilon_1}}\Big(\frac{d_0}{ H^{(1+\epsilon)/k}}\Big)^{k+\alpha-\epsilon_1-1} +  H^{1+\frac{\epsilon}{2k}}\sum_{d_0\geq H^{(1+\epsilon)/k}}\frac{1}{d_0^{k+2\alpha-2\epsilon_1}}\\
&\ll H^{\frac{1-2\alpha}{k}-\left(\frac{k+\alpha-\epsilon_1-1}{k}-\frac{\alpha-\epsilon_1}{2k^2}-\frac{1}{2k}\right)\epsilon+\frac{2\epsilon_1}{k}}.
\end{align*}

Therefore, by choosing a sufficiently small value of $\epsilon_1>0$, we have
\begin{align}\label{error term}
\mathcal{E}_{W, z}^{''}\ll 
\begin{cases}
H^{\frac{1-2\alpha}{2}-\frac{\epsilon}{3}}, & \mbox{ if }  k=2,
\\
 H^{\frac{1-2\alpha}{k}-\frac{\epsilon}{k}}, & \mbox{ if }  k\geq 3.
 \end{cases}
\end{align}

Similarly, we can simplify  $\mathcal{E}_{W, z}^{'}$  as 

\begin{align*}
& \mathcal{E}_{W, z}^{'} \ll H^{\frac{\epsilon}{k}} \sum_{d_0^k\leq H^{1-\frac{\epsilon}{2k}}}\frac{1}{d_0^{2\alpha-2\epsilon_1}}\sum_{\substack{b_1\geq H^{(1+\epsilon)/k}/d_0\\ b_2\geq 1}}\frac{1}{(b_1b_2)^{k+\alpha-\epsilon_1}}\\
&\ll H^{\frac{\epsilon}{k}} \sum_{d_0^k\leq H^{1-\frac{\epsilon}{2k}}}\frac{1}{d_0^{2\alpha-2\epsilon_1}}\Big(\frac{d_0}{ H^{(1+\epsilon)/k}}\Big)^{k+\alpha-\epsilon_1-1}\\
&\ll H^{\frac{1-2\alpha}{k}-\left(\frac{k+\alpha-\epsilon_1-1}{k}+\frac{k-\alpha +\epsilon_1}{2k^2}-\frac{1}{k}\right)\epsilon+\frac{2\epsilon_1}{k}}.
\end{align*}
Therefore, again by choosing a sufficiently small value of $\epsilon_1>0$, we have
\begin{align}\label{error term2}
\mathcal{E}_{W, z}^{'}\ll 
\begin{cases}
H^{\frac{1-2\alpha}{2}-\frac{\epsilon}{5}}, & \mbox{ if }  k=2,
\\
H^{\frac{1-2\alpha}{k}-\frac{\epsilon}{k}}, & \mbox{ if }  k\geq 3.
\end{cases}
\end{align}
 Hence, combining \eqref{error term} and \eqref{error term2}, we have
\begin{align}\label{I_W}
I_{W, z}=I_W+O\left(H^{\frac{1-2\alpha}{k}-\theta_{k, \epsilon}}\right).
\end{align}

Now we want to evaluate the term $I_W$. To simplify the expression of $I_W$, we employ the technique of contour integration. Let us define $w(x)=W(e^x)^2 e^x$, where we note that $w(x)$ is a smooth function that exhibits exponential decay as $|x|\to \infty$. Fourier inversion gives 
\begin{align}\label{hat w}
\widehat{w}(\xi)=\int_{-\infty}^\infty W(e^x)^2e^x e(-x\xi)dx=\int_0^\infty W(y)^2 y^{-2\pi i \xi}dy.
\end{align}
Integration by parts together with \ref{bounds of smoth functions} reveals that $\widehat{w}(\xi)$ is an analytic function in strip $|\Im(\xi)|<\frac{1}{2\pi}$ and uniformly in this region,
\begin{align}\label{bound of hatw}
\widehat{w}(\xi)=O\left(\frac{H^{\frac{(\lceil\beta^2/12\rceil+2)\epsilon}{2k}}}{(|\xi|+1)^{\lceil\beta^2/12\rceil+2}}\right).
\end{align}
Also, the Fourier inversion formula implies that for $r>0$,
\[
W(r)^2=\frac{1}{2\pi i r}\int_{(c)}r^s\widehat{w}\left(\frac{s}{2\pi i}\right)ds,
\]
where the integral is over $\Re(s)=c$ such that  $-1<c<1$.
Taking $c=-\frac{1}{4}$, 
\begin{align*}
I_W=\frac{H}{i\pi}\int_{(-1/4)}H^s\zeta(1-s)\bigg(\sum_{d_1, d_2\geq 1}\frac{h(d_1)h(d_2)}{d_1^kd_2^k}\frac{(d_1, d_2)^k}{(d_1, d_2)^{ks}}\bigg) \widehat{w}\left(\frac{s}{2\pi i}\right)ds. 
\end{align*}
\noindent
{\it Case 1.}
	 Assume that  $h\in \mathcal{M}_{\alpha, \beta}$.
In this case, we obtain
\begin{align*}
\sum_{d_1, d_2\geq 1}\frac{\mu(d_1)\mu(d_2)r(d_1)r(d_2)}{(d_1d_2)^{k+\alpha}}(d_1, d_2)^{k(1-s)}= B(s)\zeta^{\beta^2}(k+ks+2\alpha),
\end{align*}
where $B(s)$ is defined by \eqref{B(s)}.
Therefore, 
\begin{align}\label{Countour integral}
I_W=\frac{H}{\pi i}\int_{(-1/4)}H^s \zeta(1-s)B(s) \zeta^{\beta^2}(k+ks+2\alpha) \widehat{w}\left(\frac{s}{2\pi i}\right)ds.
\end{align}
\textbf{Step 1 {\it $($Residue estimate$)$}}
 The integrand of $I_W$ exhibits a pole of order $\beta^2$ at $s=\frac{1-2\alpha-k}{k}$, while $B(s)$ converges absolutely for $\Re(s)>\frac{1-4\alpha-2k}{2k}$. To simplify the analysis, we will shift the contour to the line $\Re(s)=\frac{3-6\alpha-4k}{4k}$ and use the decay of $\widehat{w}$. By referring to equation \eqref{G(s)}, and employing path deformation, we can express $I_W$ as follows:
\begin{align*}
I_W=2H\, \mathrm{Res}_{s=\frac{1-2\alpha-k}{k}}H^s \zeta(1-s)B(s) \zeta^{\beta^2}(k+ks+2\alpha)  \widehat{w}\left(\frac{s}{2\pi i}\right)+E,
\end{align*}
where 
\begin{align*}
E:=\frac{H}{\pi i}\int_{\big(\frac{3-6\alpha- 4k}{4k}\big)}H^s \zeta(1-s)B(s) \zeta^{\beta^2}(k+ks+2\alpha) \widehat{w}\left(\frac{s}{2\pi i}\right)ds. 
\end{align*}
For any real $t$ and $k\geq 2$, both $\zeta{\Big(\frac{8k+6\alpha-3}{4k} -it\Big)}$ and $B\Big(\frac{3-6\alpha-4k}{4k}+it\Big)$ exhibit convergence. Additionally, when $\sigma\geq \frac{1}{2}$ and $|t|\geq 2$, we have the following standard convexity bound using  $\zeta(1/2+it)\ll |t|^{1/6}\log^2 |t|$ (see \cite[eq. (8.22)]{IK}) and Phragm\'en--Lindel\"{o}f principle:
\[
\zeta(\sigma+ it)\ll \Big(|t|^{\frac{1-\sigma}{3}} +1\Big)\log{|t|}.
\]

Given that $k\geq 2$ and $\alpha<\frac12$, it follows that $\frac{4k+6\alpha-3}{8\pi k} < \frac{1}{2\pi}$. Consequently, we can deduce 
\[
\widehat{w}\left(\frac{t}{2\pi}+ i\frac{4k+6\alpha-3}{8\pi k}\right) \ll \frac{H^{\frac{(\lceil\beta^2/12\rceil+2)\epsilon}{2k}}}{(|t|+1)^{\lceil\beta^2/12\rceil+2}}. 
\]
 Thus, we have $E\ll H^{\frac{3-6\alpha}{4k}+\frac{(\lceil\beta^2/12\rceil+2)\epsilon}{2k}},$  which is  $O\left(H^{\frac{1-2\alpha}{k}-\frac{\epsilon}{k}}\right)$ when $\epsilon<\frac{1-2\alpha}{2(\lceil\beta^2/12\rceil+4)}$.
 Finally, 
 \[
 I_{W, z}=2H \,\mathrm{Res}_{s=\frac{1-2\alpha-k}{k}}H^s B(s) G_k(s)\left(s-\frac{1-2\alpha-k}{k}\right)^{-\beta^2} \widehat{w}\left(\frac{s}{2\pi i}\right)+O\left(H^{\frac{1-2\alpha}{k}-\theta_{k, \epsilon}}\right).
 \]
 \textbf{Step 2 {\it $($Estimating $S(y)$ through $W(y))$}}
Note that $S(y)$ is a smooth function but doesn't have fast decay. To overcome this limitation, we will approximate $S(y)$ using a suitable smooth function $W(y)$ obtained from the expression given in \eqref{bounds of smoth functions}. Let $v$ be a smooth bump function satisfying the conditions $v(x)=1$ when $|x|\leq 1$, and $v(x)=0$ when $|x|\geq 2$. Consider 
\[
W(y)=S(y)v\left(\frac{y}{H^{\epsilon'}}\right).
\]

Observing that $W(y)$ satisfies hypothesis \eqref{bounds of smoth functions}, and  employing the bound $h(d)\ll \frac{1}{d^{\alpha-\epsilon_1}}$, 
\begin{align}\label{compare}
&I_{S, z}-I_{W, z}=2H^2\sum_{d_1^k, d_2^k\leq z}\frac{h(d_1)h(d_2)}{d_1^kd_2^k}\sum_{\lambda\geq 1}\left(S\left(\frac{H\lambda}{(d_1^k, d_2^k)}\right)^2- W\left(\frac{H\lambda}{(d_1^k, d_2^k)}\right)^2\right)\\
&\ll  \sum_{d_1^k, d_2^k}\frac{|h(d_1)h(d_2)|}{d_1^{k}d_2^{k}}\sum_{\lambda\geq 1}\frac{(d_1, d_2)^{2k}}{\lambda^2}\mathds{1}\left(\frac{H\lambda}{(d_1^k, d_2^k)}\geq H^{\epsilon'}\right) \nonumber\\
& \ll \sum_{(d_1, d_2)\leq H^{\frac{1-\epsilon'}{k}}}\frac{(d_1, d_2)^{2k}}{(d_1d_2)^{k+\alpha-2\epsilon_1}}+ H^{1-\epsilon'}\sum_{(d_1, d_2)> H^{\frac{1- \epsilon'}{k}}}\frac{(d_1, d_2)^{k}}{(d_1d_2)^{k+\alpha-\epsilon_1}}\nonumber \ll H^{\frac{(1-\epsilon')(1-2\alpha+2\epsilon_1)}{k}}. 
\end{align}

%
%By observing that $W(y)$ satisfies hypothesis \eqref{bounds of smoth functions}, we have 
%\begin{align}\label{compare}
%&I_{S, z}-I_{W, z}=2H^2\sum_{d_1^k, d_2^k\leq z}\frac{h(d_1)h(d_2)}{d_1^kd_2^k}\sum_{\lambda\geq 1}\left(S\left(\frac{H\lambda}{(d_1^k, d_2^k)}\right)^2- W\left(\frac{H\lambda}{(d_1^k, d_2^k)}\right)^2\right)\\
%&\ll  \sum_{d_1^k, d_2^k}\frac{|h(d_1)h(d_2)|}{d_1^{k}d_2^{k}}\sum_{\lambda\geq 1}\frac{(d_1, d_2)^{2k}}{\lambda^2}\mathds{1}\left(\frac{H\lambda}{(d_1^k, d_2^k)}\geq H^{\epsilon'}\right) \nonumber\\
%& \ll \sum_{(d_1, d_2)\leq H^{\frac{1-\epsilon'}{k}}}\frac{|h(d_1)h(d_2)|(d_1, d_2)^{2k}}{(d_1d_2)^{k}}+ H^{1-\epsilon'}\sum_{(d_1, d_2)> H^{\frac{1- \epsilon'}{k}}}\frac{|h(d_1)h(d_2)|(d_1, d_2)^{k}}{(d_1d_2)^{k}}\nonumber \\
%&\ll H^{\frac{(1-\epsilon')(1-2\alpha+\epsilon_1)}{k}}, \nonumber 
%\end{align}
%where $\epsilon_1$ adjusts a finite power of $\log H$ using \eqref{r(d) bound on average}, 
and the last inequality follows from arguments similar to those above used to bound $\mathcal{E}_{W, z}$.

For every $0\leq j\leq \beta^2-1$, upon differentiating \eqref{hat w}, we obtain
\[
 \widehat{w}^{(j)}\left(\frac{1-2\alpha-k}{2\pi ik}\right)= \int_0^\infty W(y)^2 y^{\frac{k+2\alpha-1}{k}}(\log y)^j dy.
\]
The choice of $W$ yields
\begin{align*}
\int_0^\infty (W(y)^2-S(y)^2)y^{\frac{k+2\alpha-1}{k}}(\log y)^j dy\ll \int_{H^{\epsilon'}}^\infty y^{-\frac{k+1-2\alpha}{k}}(\log y)^jdy\ll H^{-\frac{(1-2\alpha)\epsilon'}{k}}(\log{H})^j.
\end{align*}
From Lemma \ref{nice int}, we deduce that 
 \begin{align}\label{approx of hat w}
 \widehat{w}^{(j)}\left(\frac{1-2\alpha-k}{2\pi ik}\right)= (-1)^{j+1} \frac{\chi^{(j)}(\frac{k+1-2\alpha}{k})}{4\pi^2} + O\left(H^{-\frac{(1-2\alpha)\epsilon'}{k}}(\log{H})^j\right).
\end{align}
Based on these estimates, 
\begin{align*}
I_{S, z}=2H \,\mathrm{Res}_{s=\frac{1-2\alpha-k}{k}}H^s B(s) G_k(s) \left(s-\tfrac{1-2\alpha-k}{k}\right)^{-\beta^2}\widehat{w}\left(\frac{s}{2\pi i}\right)+O\left(H^{\frac{1-2\alpha}{k}-\theta_{k, \epsilon}}+H^{\frac{(1-\epsilon')(1-2\alpha+\epsilon_1)}{k}}\right).
\end{align*}
By selecting $\epsilon'=\frac{\epsilon+\epsilon_1}{1-2\alpha+\epsilon_1}$, tougher with an appropriately small value of $\epsilon_1$, we conclude that 
\[
I_{S, z}= c_{h, k}P_{\beta^2-1}(\log H)\, H^{\frac{1-2\alpha}{k}}+O\left(H^{\frac{1-2\alpha}{k}-\theta_{k, \epsilon}}\right).
\]

\noindent
{\it Case 2.}
	Assume that $h$ belongs to $\mathcal{G}_{\alpha, \beta}$. We can derive the following expression:
\[
\sum_{d_1, d_2\geq 1}\frac{h(d_1)h(d_2)}{d_1^kd_2^k}\frac{(d_1, d_2)^k}{(d_1, d_2)^{ks}}=D(s) \zeta^{\beta^2}(k+ks+2\alpha),
\]
where $D(s)$ is given by \eqref{D(s)}.
Thus,
\begin{align*}
I_W=\frac{H}{i\pi}\int_{(-1/4)}H^s\zeta(1-s) D(s)\zeta^{\beta^2}(k+ks+2\alpha)\widehat{w}\left(\frac{s}{2\pi i}\right)ds.
\end{align*}
Note that the integrand in $I_W$ exhibits a pole at $s=\frac{1-2\alpha-k}{k}$ with an order of $\beta^2$. Additionally, $D(s)$ is absolutely convergent for $\Re(s)>\frac{1-4\alpha-2k}{2k}$. Consequently, we are able to shift the line of integration to $s=\frac{3-6\alpha-4k}{4k}$, and by considering Case 1, we conclude the proof. 
\end{proof}

\end{lemma}

\begin{lemma}\label{complex case main lemma}
Assume that $\beta\in \mathbb{C}\setminus \mathbb{Z}$ and $h\in \mathcal{M}_{\alpha, \beta}\cup \mathcal{G}_{\alpha, \beta}$ with $0\leq \alpha<\frac12$.
Let $S(x)$ and $I_{S, z}$ be as in Lemma \ref{lemma of main term}. 
Then for $z\geq H\exp(\sqrt{\log H})$, 
\begin{align*}
I_{S, z}=
2H^{\frac{1-2\alpha}{k}}(\log H)^{|\beta|^2-1}\bigg\{\sum_{0\leq j\leq N}\frac{k^{j-|\beta|^2}\lambda_{\alpha, j}(|\beta|^2)}{\log^j H}+O\left(\frac{1}{\log^{N+1} H}\right)\bigg\}, 
\end{align*}
where $\lambda_{\alpha, j}(|\beta|^2)$ is defined by \eqref{definition of lambda}.
\end{lemma}

\begin{proof}
Let us assume that $h\in \mathcal{M}_{\alpha, \beta}$. Following an argument of the proof of Lemma \ref{lemma of main term} and using the change of variable $z=k+ks+2\alpha$, the integral \eqref{Countour integral} can be expressed as follows:
\[
I_{W}= \frac{H^{-2\alpha/k}}{k\pi i}\int_{(3k/4+2\alpha)}L(z)\zeta^{|\beta|^2}(z)H^{\frac{z}{k}}\frac{dz}{z},
\]
where 
\[
L(z)=z\, \zeta\left(2-\frac{z-2\alpha}{k}\right)B\left(-1+\frac{z-2\alpha}{k}\right)\widehat{w}\left(\frac{-1+(z-2\alpha)/k}{2\pi i}\right).
\]
%\begin{wrapfigure}[hbt!]{r}{1.5\textwidth} %this figure will be at the right
  %  \centering
  %  \includegraphics[width=0.25\textwidth]{Contour3}
%\end{wrapfigure}
Set $F(z)=L(z)\zeta^{|\beta|^2}(z)$, $c=1+\frac{1}{\log H}$ and $T\geq 1$ is a parameter to be determined later. The integrand of $I_W$ is holomorphic within the strip $\Big\{z\in \mathbb{C}:  c\leq \Re(z)\leq \frac{3k}{4}+2\alpha\Big\}$, allowing for a smooth deformation of the line integral from $(3k/4+2\alpha)$ to $(c)$.\\
\begin{figure}[hbt!] 
	\includegraphics[scale=.8]{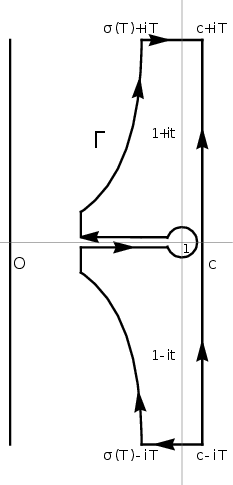}
	\caption{Contour $\Gamma$}\label{contour}
\end{figure} 
The tail parts $|\Im(z)|>T$ of $(c)$ are dominated by
\[
\ll H^{-\frac{2\alpha}{k}}\int_{c+iT}^{c+i \infty}\bigg|L(z)\zeta^{|\beta|^2}(z)H^{\frac{z}{k}}\frac{dz}{z}\bigg|\ll H^{\frac{1-2\alpha}{k}+\frac{2\epsilon}{k}}T^{-2},
\]
where we use the bound of $\widehat{w}$ from \eqref{bound of hatw}. Now, consider the contour $\Gamma$ depicted in Figure \ref{contour}, where $\sigma(t)$ is defined as $1-\frac{c_0}{\log T}$ at $t=T$ for a constant $c_0>\frac{3}{2k}$.
The contribution along the horizontal lines $[\sigma(T)\pm iT, c\pm iT]$ is bounded by
\[
\ll H^{-\frac{2\alpha}{k}+\frac{2\epsilon}{k}}T^{-3}\int_{\sigma(T)}^{c}H^{\frac{\sigma}{k}}d\sigma\ll H^{\frac{1-2\alpha}{k}+\frac{2\epsilon}{k}}T^{-3}.
\]
Further, using the standard upper bound on $\zeta(1+it)\ll \log(2+|t|)$, the contribution along the arcs $\{\sigma(t)\pm it:\, 0< t\leq T\}$ is bounded by
\[
\ll H^{\frac{\sigma(T)}{k}-\frac{2\alpha}{k}+\frac{3\epsilon}{2k}}\int_0^T \frac{(\log (2+|t|))^{|\beta|^2}}{1+|t|^3}dt\ll H^{\frac{\sigma(T)}{k}-\frac{2\alpha}{k}+\frac{3\epsilon}{2k}}.
\]
By choosing $\epsilon=\frac{1}{\sqrt{\log H}}$ and $T=H^{\epsilon}$,  the above estimates are bounded above by $O(H^{\frac{1-2\alpha}{k}}e^{-c_1\sqrt{\log H}})$ for some constant $c_1>0$.
Using the approximation of $\hat{w}$ in \eqref{approx of hat w}, along with Lemma~\ref{lemma of main term} and following the proof of Theorem $5.2$ (II) from \cite{Tenen}, we conclude that
\[
I_{W}=\frac{2}{k^{|\beta|^2}}H^{\frac{1-2\alpha}{k}}(\log H)^{|\beta|^2-1}\bigg\{\sum_{0\leq j\leq N}\frac{k^j\lambda_{\alpha, j}(|\beta|^2)}{\log^j H}+O\left(\frac{1}{\log^{N+1} H}\right)\bigg\}.
\]

The variance for class $\mathcal{G}_{\alpha, \beta}$ can be evaluated in a similar manner as $\mathcal{M}_{\alpha, \beta}$.
\end{proof}

\begin{lemma}\label{newlem1}
	Let $0\leq \alpha<2$ and $h\in \mathcal{M}_{\alpha, \beta}\cup \mathcal{G}_{\alpha, \beta}$. Recall $\mathcal{J}_{h,z}(X; H)$ from \eqref{J_k}. Then, uniformly for $z\leq X$, we have
	\begin{align*}
	\mathcal{J}_{h,z}(X; H)= (1+O(H^{-\epsilon_2}))2H^2 \sum_{d_1^k, d_2^k\leq z}\frac{h(d_1)h(d_2)}{d_1^kd_2^k}\sum_{\lambda\geq 1}S\left(\frac{H\lambda}{(d_1^k, d_2^k)}\right)^2 +O\left(\mathcal{E}_{h,z}(H)\right)+O(X^{-1+\epsilon}),
	\end{align*}
	where  
	\begin{align*}
	\mathcal{E}_{h, z}(H) := \sum_{d_1^k, d_2^k \leq z}\sum_{0<|n_1|, |n_2|\leq N}\frac{|h(d_1)h(d_2)|}{n_1n_2}\left|1-e\big(\tfrac{n_1 H}{d_1^k}\big)\right|\left|1-e\big(-\tfrac{n_2 H}{d_2^k}\big)\right|\mathbbm{1}_{0< \big|\frac{n_1}{d_1^k} - \frac{n_2}{d_2^k}\big|\leq \frac{BH^{\epsilon_2}}{X}}
	\end{align*}
	for some constant $B>0$, a sufficiently small $\epsilon$, $\epsilon_2>0$, and $N=X^4$.
\end{lemma}
\begin{proof}
Consider an integrable function $\sigma: \mathbb{R}\to \mathbb{R}$ with the property that its Fourier transform $\widehat{\sigma}$ is supported on the interval $[-BH^{\epsilon_2}, BH^{\epsilon_2}]$. Our initial goal is to demonstrate  
	\begin{small}
	\begin{align}\label{smooth version}
	\frac{1}{X}\int_{-\infty}^\infty \sigma\left(\frac{x}{X}\right)\Bigg|\sum_{\substack{x<nd^k\leq x+H\\ d^k\leq z}}h(d)-H\sum_{d^k\leq z}\frac{h(d)}{d^k}\Bigg|^2 dx\\
	=2H^2 \sum_{d_1^k, d_2^k\leq z}\frac{h(d_1)h(d_2)}{d_1^kd_2^k}\sum_{\lambda\geq 1}S\left(\frac{H\lambda}{(d_1^k, d_2^k)}\right)^2 + O(\mathcal{E}_{k,h}(H))+O(X^{-1+\epsilon}).\nonumber
	\end{align}
\end{small}	
	Observe that
	\begin{align*}
	\sum_{x/d^k< n\leq (x+H)/d^k}1= \frac{H}{d^k}+\psi\left(x/d^k\right)-\psi\left((x+H)/d^k\right),
	\end{align*}
	where $\psi(y)=\{y\}-\frac{1}{2}$ with $\{y\}$ the fractional part of $y$. The Fourier expansion of $\psi$ is (see \cite[eq.  (4.18)]{IK})
	\begin{align*}
	\psi(y)=-\frac{1}{2\pi i} \sum_{0<|n|\leq N}\frac{e(yn)}{n}+O\left(\min\bigg\{1, \frac{1}{N\|y\|}\bigg\}\right).
	\end{align*} 
	So, combining the above estimates, we have
	\begin{align*}
	\sum_{\substack{x<nd^k\leq x+H\\ d^k\leq z}}h(d)-H\sum_{d^k\leq z}\frac{h(d)}{d^k}=-\frac{1}{2\pi i}\sum_{d^k \leq z}h(d)\sum_{0<|n|\leq N}\frac{1}{n}e\left(\frac{nx}{d^k}\right)\left(1-e\left(\frac{nH}{d^k}\right)\right)\\
	+ O\bigg(\sum_{d^k\leq z}h(d)\left(\min\{1, 1/(N\|x/d^k\|)\}+\min\{1, 1/(N\|(x+H)/d^k\|)\}\right)\bigg).
	\end{align*}
Since \( z \leq X \), by choosing \( N = X^4 \), the minimum values inside the error term above can be bounded by \( O(X^{-2}) \), unless 
\[
\left| \frac{x}{d^k} \right| < X^{-2} \quad \text{or} \quad \left| \frac{x+H}{d^k} \right| < X^{-2}.
\]
However, both of these inequalities lead to contradictions and are therefore impossible.

	Therefore, the contribution of the error term to the integral \eqref{smooth version} is bounded above by
	\[
	\ll \frac{X^{2(1-\alpha)/k+2\epsilon_1}\log X}{X^2},
	\qquad \text{if } k\geq 2,
	\]
	and by
	\[
	\ll \frac{X^{2(1-\alpha)+2\epsilon_1}}{X^2},
	\qquad \text{if } k=1.
	\]
	Hence, in both cases,
	\[
	\ll X^{-1+\epsilon},
	\]
	since \(k+\alpha>3/2\) for a suitable choice of \(\epsilon>0\).

  Referring to \eqref{smooth version}, we can simplify our analysis to focus on the following expression, accompanied by a negligible error.
	\begin{align}\label{term after inverse fourier}
	\frac{1}{4\pi^2}\sum_{\substack{d_1^k, d_2^k \leq z\\ 0<|n_1|, |n_2|\leq N }}\frac{h(d_1)h(d_2)}{n_1n_2}\left(1-e\left(\frac{n_1 H}{d_1^k}\right)\right)\left(1-e\left(-\frac{n_2 H}{d_2^k}\right)\right) \widehat{\sigma}\left(X\left(\frac{n_2}{d_2^k}-\frac{n_1}{d_1^k}\right)\right). 
	\end{align}
	We analyze two separate cases: first, those $(n_1, n_2, d_1, d_2)$ where $n_1d_2^k - n_2d_1^k = 0$, and second, those where this equation does not hold.
	
	When $n_1d_2^k - n_2d_1^k \neq 0$, the aforementioned expression is bounded above by $\mathcal{E}_{h, z}(H)$. In the former scenario, we may express $n_1$ and $n_2$ as $n_1=\frac{\lambda d_1^k}{(d_1^k, d_2^k)}$ and $n_2=\frac{\lambda d_2^k}{(d_1^k, d_2^k)}$, where $\lambda\in \mathbb{Z}\backslash \{0\}$. Using the identity $|1-e(x)|=2|\sin (\pi x)|$, we obtain the main contribution in  \eqref{term after inverse fourier} as
	\begin{align*}
	\frac{\widehat{\sigma}(0)}{\pi^2}\sum_{d_1^k, d_2^k \leq z}\frac{h(d_1)h(d_2)}{d_1^k d_2^k}(d_1, d_2)^{2k}\sum_{\lambda\neq 0}\frac{1}{\lambda^2}\sin^2\left(\frac{\lambda \pi H}{(d_1^k, d_2^k)}\right).
	\end{align*}
	
	In order to establish a connection between $\mathcal{J}_{h, z}(X; H)$ and the integral \eqref{smooth version}, we employ the argument presented in \cite[page 273]{Montgomery}, which is also utilized in \cite{GMRR}. Following this approach, we can obtain a pair of smooth, integrable functions $\sigma{-}$ and $\sigma_{+}$ whose Fourier transforms $\widehat{\sigma}{-}$ and $\widehat{\sigma}{+}$ have a support interval of $[-BH^{\epsilon_2}, BH^{\epsilon_2}]$. These functions satisfy 
	\begin{align*}
	\sigma_{-} \leq \mathbbm{1}_{[1,\, 2]} \leq \sigma_{+}\, \mbox{ and }\, \Big|\int \Big(\sigma_{\pm}(x)-\mathbbm{1}_{[1,\, 2]}(x)\Big)dx\Big| \ll H^{-\epsilon_2}.
	\end{align*}  
	
	\noindent
	Thus, we conclude that 
	\begin{align*}
	\mathcal{J}_{h, z}(X; H)=(1+O(H^{-\epsilon_2}))\frac{1}{X}\int_{-\infty}^\infty \sigma_{\pm}\left(\frac{x}{X}\right)\Bigg|\sum_{\substack{x<nd^k\leq x+H\\ d^k\leq z}}h(d)-H\sum_{d^k\leq z}\frac{h(d)}{d^k}\Bigg|^2 dx.
	\end{align*}
	\end{proof}

\subsection{Density of rational points on binary forms}\label{6.2}
In this section, we will discuss the important lemmas related to the density of rational points on binary forms. These lemmas play a crucial role in bounding the non-diagonal terms mentioned in Lemma \ref{newlem1}. Let $F(x, y)$ be a binary form with integer coefficients, non-zero discriminant, and degree $d\geq 2$.  For a positive number $Z$, let
\[
\mathcal{N}_F(Z):=\{(x, y)\in \mathbb{Z}^2 : \, 0< |F(x, y)|\leq Z\}
\]
and 
\[
A_F= \mu\{(x, y)\in \mathbb{R}^2: \, |F(x, y)|\leq 1\},
\]
where $\mu$ denotes the area of a set in $\mathbb{R}^2$. In 1933, Mahler~\cite{Mahler} derived an asymptotic formula for the cardinality $|\mathcal{N}_F(Z)| \sim A_F Z^{\frac{2}{d}}$, when $F$ is irreducible. 
However, the absolute constant in the error term depends on the form $F$. To achieve our result we need an upper bound of $|\mathcal{N}_F(Z)|$, which is uniform on the coefficients of $F$. In this direction, we refer to a result by Schmidt \cite{SCHM} from 1987 in the following lemma.
\begin{lemma}\label{a result of Schmidt}
Let $F$ be a binary form with integer coefficients of degree $d \geq 3$ which is irreducible over the rationals. Suppose that not more than $s+1$ of the coefficients are non-zero. Then we have
\[
|\mathcal{N}_F(Z)| \ll \sqrt{sd} Z^{\frac{2}{d}} \bigg(1+ \frac{1}{d}\log{Z}\bigg),
 \]
where the constant implied by $\ll$ is absolute.
\end{lemma}
When counting solutions of large size, we should expect a stronger upper bound on the number of solutions.
Following Heath-Brown’s proof of Theorem 8 in~\cite{Heath}, Stewart and Xiao~\cite[Lemma 2.1]{SX} recently proved the following.
\begin{lemma}\label{a result of Stewart and Xiao}
Let $F$ be a binary form with integer coefficients, non-zero discriminant, and degree $d \geq 3$. Let $Z$ be a positive real number, and $\gamma$ be a real number larger than $\frac{1}{d}$. The number of pairs of integers $(x, y)$ with
\[
0 < |F(x, y)| \leq Z
\]
for which
$
\max\{|x|, |y|\}>Z^\gamma
$
is
\[
O_F\left(Z^{\frac{1}{d}}\log Z + Z^{1-(d-2)\gamma}\right).
\]
\end{lemma}  
For our application, it is necessary to ensure that the constant implied by the symbol $\ll$ in Lemma~\ref{a result of Stewart and Xiao} is either absolute or depends only on $d$. In their lemma, Stewart and Xiao noted that, rather than relying on Roth’s theorem as in Heath-Brown’s argument, one can handle the large solutions using the Thue–Siegel principle. This alternative approach leads to an effective constant. In what follows, we establish such a result by combining Heath-Brown’s method with a result of Győri~\cite{GY}, which is based on the Thue–Siegel principle.

\begin{lemma}\label{efficitive bound for large solutions}
Let $a, b\in \mathbb{Z}\setminus \{0\}$, and $F(x, y)=ax^d-by^d$ with $d\geq 3$.  Let $Z$ be a positive real number, and $\gamma$ be a real number larger than $\frac{1}{d}$. Then the number of pairs of integers $(x, y)$ with
\[
0 < |F(x, y)| \leq Z
\]
for which
$
\max\{|x|, |y|\}>Z^\gamma
$
is
\[
O_d \left(Z^{\frac{1}{d}} + \frac{1}{\min(|a|, |b|)}Z^{1-(d-2)\gamma}\right)\log{(|ab|Z)},
\]
where the absolute constant depends only on $d$.
\end{lemma}
\begin{proof}
For the shake of simplicity, it is enough to consider the case when $a\geq b\geq 1$. Let us define the size of the set of large solutions by
\begin{align*}
N(Z, \gamma):= |\{ (x, y)\in \mathbb{Z}^2 : 0 < |F(x, y)| \leq Z, \max\{|x|, |y|\}>Z^\gamma \}|,
\end{align*}
and the size of the set of primitive large solutions by
\begin{align*}
S(Z, Z^\gamma):= |\{ (x, y)\in \mathbb{Z}^2 : 0 < |F(x, y)| \leq Z, \max\{|x|, |y|\}>Z^\gamma, (x,y)=1 \}|.
\end{align*}
We can relate them by the following equation:
\begin{align*}
N(Z, \gamma)= \sum_{n=1}^{Z^{\frac{1}{d}}}S\bigg( \frac{Z}{n^d}, \frac{Z^\gamma}{n}\bigg).
\end{align*}
For $C\geq Z^\gamma$, we denote the size of the set of dyadic primitive solutions by
\begin{align*}
S_1(Z, C):= |\{ (x, y)\in \mathbb{Z}^2 : 0 < |F(x, y)| \leq Z,C\leq  \max\{|x|, |y|\}\leq 2C, (x,y)=1 \}|.
\end{align*}
Then for a parameter $Y\geq C$ to be chosen later,  
\begin{align*}
S(Z, C) = \sum_{j=0}^{\infty}S_1\big( Z, 2^jC\big) \leq  \sum_{j=0}^{\lceil\log{(Y/C)}/\log 2\rceil +1}S_1\big( Z, 2^jC\big) + S(Z, Y).
\end{align*}
A result of Győri~\cite[Theorem 1]{GY} gives under the choice of
\begin{align}\label{choice of Y}
 Y= \bigg(  \frac{1}{\sqrt{|\Delta|}} 2^{d+1}d^{d/2}M^d Z\bigg)^{\frac{1}{d-2}+\frac{1}{d^2}}= \bigg(  2^{d+1}a^{\frac{d+1}{2}}b^{\frac{1-d}{2}}  Z\bigg)^{\frac{1}{d-2}+\frac{1}{d^2}},
 \end{align}that
 \[S(Z, Y)\leq 25d,\]
 where $M=a$ is the Mahler measure of $F(x,1)$, and $\Delta = (-1)^{\frac{d(d-1)}{2}}d^d(ab)^{d-1}$, the discriminant of the form $F(x, y)$.
 Hence we have 
\begin{align}\label{S(Z,C)}
S(Z, C)\leq  \left(\lceil\log{(Y/C)}/\log 2\rceil +1\right) \max_{0\leq j\leq \lceil\log{(Y/C)}/\log 2\rceil +1} S_1\big( Z, 2^jC\big) + 25d.
\end{align}
Letting  $C_1=2^jC$ for any fixed $j$ in the range, we need an upper bound for $S_1\big( Z, C_1\big)$.

Set $Z_1 := Z/a$ and $\alpha_j=(b/a)^{1/d}e(j/d), \quad j=1,2,\ldots,d$.
\begin{case}[$d$ is fixed and even] It therefore suffices to count the solutions in \(\mathbb{N}^2\), since the total number of solutions is exactly four times the number of solutions in this region.
	
We notice that
\begin{align*}
|x-\alpha_j y|^2 & = x^2 + y^2(b/a)^{2/d} -2xy(b/a)^{1/d}\cos\left(\frac{2\pi j}{d}\right) = (x-\alpha_dy)^2 + 4xy(b/a)^{1/d}\sin^2\left(\frac{\pi j}{d}\right)\\
&= (x-\alpha_dy)^2\cos^2\left(\frac{\pi j}{d}\right) + (x+\alpha_dy)^2\sin^2\left(\frac{\pi j}{d}\right) \geq (x+\alpha_dy)^2\sin^2\left(\frac{\pi j}{d}\right).
\end{align*}
Thus we have
\begin{align*}
|x-\alpha_j y| > (x+\alpha_dy) \bigg|\sin\left(\frac{\pi j}{d}\right)\bigg|=C_1\bigg( \frac{x}{C_1}+\alpha_d\frac{y}{C_1}\bigg) \bigg|\sin\left(\frac{\pi j}{d}\right)\bigg|.
\end{align*}
Note that either $\frac{x}{C_1}\geq 1$ or $\frac{y}{C_1}\geq 1$. Since $ \big|\sin(\frac{\pi j}{d})\big|>0$ is a constant for all $j\neq d$, we have 
\begin{align}\label{08/25}
|x-\alpha_j y| \gg_d \begin{cases}  C_1 & \mbox{ if } x \geq C_1,\\
 C_1\alpha_d & \mbox{ if } y \geq C_1.
\end{cases}
\end{align}
\end{case}

\begin{case}[$d$ is fixed and odd]
In this case, whenever the components of a primitive solution have the same sign, we can establish \eqref{08/25} using the same argument as in Case~1. However, if the components have opposite signs, we proceed as follows:
\begin{align*}
|x-\alpha_j y|^2 & = x^2 + y^2(b/a)^{2/d} + 2|x||y|(b/a)^{1/d}\cos\left(\frac{2\pi j}{d}\right) \\
&= (|x|+\alpha_d|y|)^2\cos^2\left(\frac{\pi j}{d}\right) + (|x|-\alpha_d|y|)^2\sin^2\left(\frac{\pi j}{d}\right)\\
&\geq (|x|+\alpha_d|y|)^2\cos^2\left(\frac{\pi j}{d}\right).
\end{align*}
Since \( d \) is fixed and odd, \( \big|\cos\left(\frac{\pi j}{d}\right)\big| > 0 \) is bounded below by a positive constant for all \( j \), and the conclusion~\eqref{08/25} follows.
\end{case}
Thus, in both cases, we obtain 
\begin{align*}
|x-\alpha_d y|C_1^{d-1}\alpha_d^{d-1}\ll |x-\alpha_d y|\prod_{j\neq d}|x-\alpha_j y|\ll Z_1.
 \end{align*}
This implies that \[  |x-\alpha_d y|\ll_d \frac{Z_1}{\alpha_d^{d-1}C_1^{d-1}}.\]
Heath-Brown observed (see the proof of \cite[Theorem~8]{Heath}) that an upper bound for
\(|S_1(Z, C_1)|\) is given by the number of lattice points in the parallelogram
\[
|y|\leq 2C_1,
\qquad
|x-\alpha_d y|\ll_d \frac{Z_1}{\alpha_d^{d-1}C_1^{d-1}}.
\]
This region has area, say \(A(Z)\), satisfying
\[
A(Z)\ll_d \frac{Z_1}{\alpha_d^{d-1}C_1^{d-2}}.
\]
We then enclose this parallelogram in a rectangle centered at the origin with same area and the same number of lattice points. Finally, this rectangle can be further embedded in an ellipse whose area is comparable up to an absolute constant (see also the proof of \cite[Lemma~1, part (vii)]{Heath}).

 Thus, by Lemma 1 of \cite{HB1},  we have 
\begin{align}\label{S_1(Z,C)}
|S_1\big( Z, C_1\big)|\ll_d 1+ \frac{A(Z)}{\alpha_d}\ll 1+  \frac{Z_1}{\alpha_d^{d}C_1^{d-2}} \ll_d 1 + \frac{Z}{bC_1^{d-2}}.
\end{align}
By combining \eqref{S_1(Z,C)} with \eqref{S(Z,C)} with the above choice of $Y$ from \eqref{choice of Y}, we have 
\begin{align*}
S(Z, C)\ll_d \bigg( 1 + \frac{Z}{bC^{d-2}} \bigg) \log{(abZ)}.
\end{align*}
This gives
\begin{align*}
N(Z, \gamma)= \sum_{n=1}^{Z^{\frac{1}{d}}}S\bigg( \frac{Z}{n^d}, \frac{Z^\gamma}{n}\bigg)\ll_d  \bigg( Z^{\frac{1}{d}} + \frac{Z^{1-\gamma(d-2)}}{b} \bigg) \log{(abZ)}.
\end{align*}
If $a\leq b$, then we can have the above estimate with $b$ replaced by $a$. Hence the result follows.
\end{proof}

As an application of the above lemmas, we bound the term $\mathcal{E}_{h, z}(H)$ as follows.
\begin{lemma}\label{E(H)}
Let $\epsilon>0$ be given and suppose that $1<\nu\leq 2$. Recall $\mathcal{E}_{h, z}(H)$ from Lemma \ref{newlem1}. 
For $0\leq \alpha<\frac12$, we have
\begin{align*}
\mathcal{E}_{h, z}(H)\ll  \begin{cases}
H^{\frac{1-2\alpha}{2}-\frac{\epsilon}{5}} & \mbox{ if }  z\leq \min\big\{X^{\frac{1}{(1-\alpha)}}H^{-\frac{1+2\alpha}{2-2\alpha}-\epsilon}, H^{\frac{1-2\alpha}{2-2\alpha}- \epsilon}X^{\frac{1}{2-2\alpha}}\big\}, \,  H\leq X^{\frac{2+\alpha}{3+2\alpha}-\epsilon},\, k=2,\\
H^{\frac{1-2\alpha}{k}-\frac{\epsilon}{k}} & \mbox{ if } z\leq \min \big\{X^{\frac{2}{\nu(2-\alpha)}}H^{\frac{1-2\alpha}{\nu(2-\alpha)}-\epsilon}, X^{\frac{\nu}{2(\nu-\alpha)}}H^{\frac{1-2\alpha}{2(\nu-\alpha)}-\epsilon}, X^{\frac{2(k-\nu)}{(4-\nu)k-2\nu}-\epsilon}\big\},\, k\geq 3.
\end{cases}
\end{align*}
For $\frac{1}{2}<\alpha<1$, let  $\epsilon_{\alpha, k}>0$ be sufficiently small, depending on $\epsilon, \alpha$ and $k$. Then we obtain 
\[\mathcal{E}_{h, z}(H)\ll H^{-\epsilon_{\alpha, k}},\] whenever
\begin{align*}
z\leq 
  \begin{cases}
 X^{\frac{1}{2(2-\alpha)}}H^{\frac{1}{2(2-\alpha)}-\epsilon} & \mbox{ if }  k=1,\\
  \min\big\{X^{\frac{1}{1-\alpha}}H^{-\frac{1}{1-\alpha}-\epsilon}, X^{\frac{1}{2-2\alpha}}H^{-\epsilon}\big\} & \mbox{ if } H\leq X^{\frac{2+\alpha}{4}-\epsilon} \text{ and }  k=2,\\
 \min \big\{X^{\frac{2}{\nu(2-\alpha)}-\epsilon}, X^{\frac{\nu}{2(\nu-\alpha)}-\epsilon}, X^{\frac{2(k-\nu)}{(4-\nu)k-2\nu}-\epsilon}\big\} & \mbox{ if } k\geq 3.
\end{cases}
\end{align*}
Suppose that $1\leq \alpha< 2$. In this case, we have $\mathcal{E}_{h, z}(H)\ll H^{-\epsilon_{\alpha, k}}$, provided that
\begin{align*}
z\leq 
\begin{cases}
X^{\frac{1}{2(2-\alpha)}}H^{\frac{1}{2(2-\alpha)}-\epsilon} & \mbox{ if }  k=1,\\
X^{1-\epsilon} & \mbox{ if } H\leq X^{\frac{2+\alpha}{4}-\epsilon} \text{ and }  k=2,\\
X^{1-\epsilon} & \mbox{ if } k\geq 3.
\end{cases}
\end{align*}
\end{lemma}

\begin{proof} We will prove the lemma by considering different values of $k$ in separate cases.\\
{\it Case $1$.}
Let us assume that $k=1$. By splitting $n_j$ and $d_j$ dyadically, for any $D_1, D_2 \leq z$ and any $N_1, N_2 \leq N$, we have
\begin{align*}
\mathcal{E}_{h, z}(H) &\ll z^{2\epsilon_1}\max_{\substack{D_1, D_2\leq z\\
		N_1, N_2\leq N}}\frac{(\log X)^4}{(D_1D_2)^{\alpha}}\min\bigg\{\frac{1}{N_1}, \frac{H}{D_1}\bigg\}\min\bigg\{\frac{1}{N_2}, \frac{H}{D_2}\bigg\}\\
& \times \sum_{\substack{n_1\sim N_1\\ n_2\sim N_2}}\# \Big\{d_j\sim D_j:\, 0<\bigg|\frac{n_1}{d_1}- \frac{n_2}{d_2}\bigg|\leq \frac{BH^{\epsilon_2}}{X}\Big\}
\end{align*}
since $r(D)\ll D^{\epsilon_1}\ll z^{\epsilon_1}$.
Note that, a solution of the above system exists only if $N_1D_2\asymp N_2D_1$. In this case, the innermost sum is equivalent to 
\[
\# \Big\{r_1 \sim N_2D_1; r_2\sim N_1D_2:\, 0<|r_1- r_2|\leq Y\Big\},
\]
where $Y=\frac{BH^{\epsilon_2}D_1 D_2}{X}$. Observe that 
\[
\# \Big\{r_1 \sim N_2D_1; r_2\sim N_1D_2:\, 0<|r_1- r_2|\leq Y\Big\}\ll \max\{N_2D_1 Y,\, N_1D_2 Y,\, Y^2\}. 
\]
The bound of $\mathcal{E}_{h, z}(H)$ attains its maximum provided that $N_j=\frac{D_j}{H}$ for $j=1, 2$.
Therefore, for $\alpha\in (1/2, 2)$,
\begin{align*}
\mathcal{E}_{h, z}(H) 
& \ll z^{2\epsilon_1} H^{\epsilon_2}\max_{\substack{D_1, D_2\leq z}}\frac{(D_1D_2)^{2-\alpha}}{HX}(\log X)^4 \ll  
H^{-\epsilon_{\alpha, 1}},
\end{align*}
provided that $z\leq X^{\frac{1}{2(2-\alpha)}}H^{\frac{1}{2(2-\alpha)}-\epsilon}$, by taking $\epsilon_1$ and $\epsilon_2$ suitably small.
 
\vspace{2mm}
\noindent
{\it Case $2$.}\label{sub2.1}
Assume that $k=2$ and $0\leq \alpha<\frac12$. We choose $\epsilon_2=\frac{\epsilon}{4}$ and $\epsilon_1\ll \epsilon^2$. By using the inequality $h(d)\ll \frac{1}{d^{\alpha-\epsilon_1}}$, we can follow the proof of Proposition $5$ \cite[eq. (32)]{GMRR} to obtain  $\mathcal{E}_{h, z}(H) \ll H^{\frac{1-2\alpha}{2}-\frac{\epsilon}{5}}$ only when  $z\leq \min\Big\{X^{\frac{1}{1-\alpha}}H^{-\frac{1+2\alpha}{2-2\alpha}-\epsilon}, X^{\frac{1}{2-2\alpha}}H^{\frac{1-2\alpha}{2-2\alpha}- \epsilon}\Big\}$ and $H\leq X^{\frac{2+\alpha}{3+2\alpha}-\epsilon}$.

Now assuming $\frac{1}{2}<\alpha< 1$, we again follow the proof of Proposition $5$ \cite[eq. (32)]{GMRR}. 
The important observation is that
\[
\# \Big\{(d_1 \sim D_1; d_2\sim D_2:\, 0<|n_1d_2^2- n_2 d_1^2|\leq \frac{BH^{\epsilon_2}D_1^2 D_2^2}{X}\Big\}\neq 0
\]
whenever $D_1D_2\gg X^{1/2}H^{-\epsilon_2/2}$.

If $\sqrt{\frac{n_2}{n_1}}$ is a quadratic irrational, then $\mathcal{E}_{h, z}(H)$ is bounded above by
 \begin{align}\label{16.4.4.15}
 \ll z^{2\epsilon_1}(\log X)^4 \max_{\substack{D_1, D_2\leq z^{1/2}}}\frac{H^{\epsilon_2}}{D_1^{\alpha} D_2^{\alpha}}\left(\frac{D_1D_2H}{X}+1+\frac{D_1D_2}{X^{1/2}}\right)\ll H^{-\epsilon_{\alpha, k}},
 \end{align}
  whenever 
 $z\leq \min\Big\{X^{\frac{1}{1-\alpha}}H^{-\frac{1}{1-\alpha}-\epsilon}, X^{\frac{1}{2-2\alpha}}H^{-\epsilon}\Big\}$.

 Additionally, if $\sqrt{\frac{n_2}{n_1}}$ is rational, then we have 
 \begin{align}\label{rational}
 \mathcal{E}_{h, z}(H)\ll z^{\epsilon_1}\max_{\substack{D_1, D_2\leq z^{1/2}}}\frac{H^2}{(D_1 D_2)^{\alpha}X}\ll H^{-\epsilon_{\alpha, 2}},
 \end{align}
 since $H\leq X^{\frac{2+\alpha}{4}-\epsilon}$ and $D_1D_2\gg X^{1/2}H^{-\epsilon_2/2}$.

For $1\leq \alpha< 2$, we get from \eqref{16.4.4.15} and \eqref{rational} that $\mathcal{E}_{h, z}(H)\ll H^{-\epsilon_{\alpha, 2}}$, since $H\leq X^{\frac{2+\alpha}{4}-\epsilon}$.

\vspace{2mm}
\noindent
{\it Case $3$.}
Let $k\geq 3$.  First assume that $1\leq \alpha< 2$. Since $z\leq X^{1-\epsilon}$, Lemma \ref{a result of Schmidt} yields 
\begin{align*}
\mathcal{E}_{h, z}(H)\ll z^{2\epsilon_1}(\log X)^4\max_{\substack{D_1, D_2\leq z^{1/k}}}(D_1D_2)^{-\alpha}\left(\frac{H^{\epsilon_2}D_1^kD_2^k}{X}\right)^{\frac{2}{k}}\ll H^{-\epsilon_{\alpha, k}}.
\end{align*}

Now, we consider $ \alpha\in [0, 1)\setminus \{1/2\}$. Let us define  
\[
\mathcal{R}:=\# \Big\{d_1\sim D_1, d_2\sim D_2:\, 0<|n_1d_2^k- n_2 d_1^k|\leq \frac{BH^{\epsilon_2}D_1^k D_2^k}{X}\Big\}.
\] 
For $1<\nu \leq 2$, we divide $\mathcal{R}$ into two parts as follows   
\[
\mathcal{R}_{\nu}(n_1, n_2):= \# \Big\{d_j\sim D_j:\, D_j \leq z^{1/k}, \, D_1D_2< z^{\nu/k}, \, 0<|n_1d_2^k- n_2 d_1^k|\leq \frac{BH^{\epsilon_2}D_1^k D_2^k}{X}\Big\},
\]
and 
\[
\mathcal{R}_{\nu}^C(n_1, n_2):= \# \Big\{d_j\sim D_j:\, D_j \leq z^{1/k}, \, D_1D_2\geq z^{\nu/k}, \, 0<|n_1d_2^k- n_2 d_1^k|\leq \frac{BH^{\epsilon_2}D_1^k D_2^k}{X}\Big\}.
\]
Therefore, 
\begin{align*}
\mathcal{E}_{h, z}(H) \ll z^{2\epsilon_1}\max_{\substack{D_1, D_2\leq z^{1/k}}}\frac{(\log X)^2}{(D_1D_2)^{\alpha}}\sum_{n_j\leq N}\frac{1}{n_1 n_2}\left(\mathcal{R}_{\nu}(n_1, n_2)+ \mathcal{R}_{\nu}^C(n_1, n_2)\right).
\end{align*}
We utilize the lemmas introduced at the start of this sub-section to establish upper limits for the above quantities.
Lemma \ref{a result of Schmidt} gives us
\[
\mathcal{R}_{\nu}(n_1, n_2)\ll_k \left(\frac{H^{\epsilon_2}D_1^k D_2^k}{X}\right)^{\frac{2}{k}}\log z,
\]
where the implied constant in $\ll$ depends only on $k$.

Let us define $Z=\frac{BH^{\epsilon_2}D_1^k D_2^k}{X}$. Given the range of $d_1$ and $d_2$, it follows that if $D_1D_2> z^{\nu/k}$, then $\max\{D_1, D_2\}> z^{\nu/2k}$. In light of Lemma \ref{efficitive bound for large solutions}, where $\gamma=\frac{1}{k}+\frac{2-\nu}{k(k-2)}$, we have
\begin{align*}
\mathcal{R}_{\nu}^C(n_1, n_2)&\ll \# \Big\{(d_1, d_2)\in \mathbb{Z}^2:\, \max\{D_1, D_2\}>Z^{\gamma}, \, 0<|n_1d_2^k- n_2 d_1^k|\leq Z\Big\} \\
&\ll \left(Z^{1/k}+ Z^{1-(k-2)\gamma}\right)\log(|n_1 n_2| X)\ll Z^{\frac{\nu}{k}} \log X,
\end{align*}
since $Z^{\gamma}\leq z^{\frac{\nu}{2k}}$, and $1\leq |n_1|, |n_2|\leq X^{4}$. 
The condition $Z^{\gamma}\leq z^{\frac{\nu}{2k}}$ is equivalent to $z\leq (XH^{-\epsilon_2})^{\frac{2(k-\nu)}{(4-\nu)k-2\nu}}$.

Therefore, when $z\leq (XH^{-\epsilon_2})^{\frac{2(k-\nu)}{(4-\nu)k-2\nu}}$, we can combine these estimates to find that
\begin{align*}
\mathcal{E}_{h, z}(H)\ll z^{2\epsilon_1} (\log X)^5\Bigg(\max_{\substack{D_1, D_2\leq z^{1/k}\\ D_1 D_2\leq z^{\nu/k}}}\frac{1}{(D_1D_2)^{\alpha}}\left(\frac{H^{\epsilon_2}D_1^k D_2^k}{X}\right)^{\frac{2}{k}}+\max_{\substack{D_1, D_2\leq z^{1/k}\\ z^{\nu/k}<D_1D_2< z^{2/k}}}\frac{Z^{\frac{\nu}{k}}}{(D_1D_2)^{\alpha}}\Bigg).
\end{align*}

\medskip
\noindent{\it Subcase 3.1.}
Assume that $0\leq \alpha<\frac12$. Choosing $\epsilon_2=\frac{\epsilon}{k}$ and $\epsilon_1\ll \epsilon^2$, we derive 
\[ 
\mathcal{E}_{h, z}(H)\ll H^{\frac{1-2\alpha}{k}-\frac{\epsilon}{k}},
\]
provided that $z\leq \min \Big\{X^{\frac{2}{\nu(2-\alpha)}}H^{\frac{1-2\alpha}{\nu(2-\alpha)}-\epsilon}, X^{\frac{\nu}{2(\nu-\alpha)}}H^{\frac{1-2\alpha}{2(\nu-\alpha)}-\epsilon}, X^{\frac{2(k-\nu)}{(4-\nu)k-2\nu}-\epsilon}\Big\}$
and $\nu> 1.4$.

\medskip
\noindent{\it Subcase 3.2.}
Assume that $\frac12<\alpha<1$. Under the condition
\[
z\leq
\min\Big\{
X^{\frac{2}{\nu(2-\alpha)}-\epsilon},
\,
X^{\frac{\nu}{2(\nu-\alpha)}-\epsilon},
\,
X^{\frac{2(k-\nu)}{(4-\nu)k-2\nu}-\epsilon}
\Big\},
\]
we have
$
\mathcal{E}_{h,z}(H)
\ll
H^{-\epsilon_{\alpha,k}},
$ which completes the proof.
\end{proof}

\begin{proof}[Proof of Proposition \ref{below z-barrier}]
 The proof is completed by combining Lemmas \ref{lemma of main term}, \ref{newlem1}, and  \ref{E(H)}.
\end{proof}

\begin{proof}[Proof of Proposition \ref{proposition for lower growth family}]
The proof begins by considering $f \in \mathcal{F}_{\alpha,\beta, k}$ with $\alpha > \frac{1}{2}$. Recall $I_W$ from \eqref{I_W} and define $I_{S}$ as
\begin{align*}
I_{S}:
 = 2H^2\sum_{d_1, d_1\geq 1}\frac{h(d_1)h(d_2)}{d_1^kd_2^k}\sum_{\lambda\geq 1}S\left(\frac{H\lambda}{(d_1^k, d_2^k)}\right)^2.
\end{align*}
By following the proof of Lemma \ref{lemma of main term}, we are only left to show that $I_S-I_W$ is bounded by $O(H^{-\epsilon_{\alpha, k}})$. Using the bound in \eqref{compare}, we have
\begin{align*}
|I_S-I_W|&\ll  \sum_{d_1, d_2\geq 1}\frac{h(d_1)h(d_2)}{d_1^{k}d_2^{k}}\sum_{\lambda\geq 1}\frac{(d_1, d_2)^{2k}}{\lambda^2}\mathds{1}\left(\frac{H\lambda}{(d_1^k, d_2^k)}\geq H^{\epsilon'}\right)\ll H^{\frac{(1-\epsilon')(1-2\alpha+2\epsilon_1)}{k}}\ll H^{-\epsilon_{\alpha, k}},
\end{align*}
 for a suitably small choice of $\epsilon'$ and $\epsilon_1$. 
Since $2\sin^2{x}= 1 - \cos{2x}$, by using Lemma \ref{lem1}, 
\begin{align*}
I_{S, z}=\sum_{d_1,d_2 \geq 1} \frac{h(d_1)h(d_2)}{d_1^k d_2^k}(d_1, d_2)^{2k} \left( \left\{\ \frac{H}{(d_1, d_2)^k} \right\} - \left\{\ \frac{H}{(d_1, d_2)^k} \right\} ^2\right) +O\left(H^{-\epsilon_{\alpha, k}}\right).
\end{align*}
Introducing the variable $d$ as $d:= (d_1, d_2)$ and writing $d_1=e_1 d$, $d_2=e_2 d$, we have $(e_1, e_2)=1$. With this change of variables, we can simplify and reduce the main term of the above expression to $c_{h, k}(H)$, which is the main term of this proposition (see eq. \eqref{C_h,k}). Finally, Lemmas~\ref{newlem1} and \ref{E(H)} complete the proof of the proposition. 
\end{proof}

\section{Proof of  Proposition \ref{19.04.22.6:26}}\label{section 6}
We collect several fundamental lemmas from \cite{GMRR} that serve as the main tools for proving the proposition.
\begin{lemma}\label{large value theorem}
Assume $N, T\geq 1$ and $V>0$. Let $A(s)=\sum_{n\leq N}a_n n^{-s}$ be a Dirichlet polynomial and let $S=\sum_{n\leq N}|a_n|^2$. Suppose $\mathcal{T}$ is a set of $1$-spaced points $t_r\in [-T, T]$ such that $|A(it_r)|\geq V$. Then
\[
|\mathcal{T}|\ll \left(SNV^{-2}+T\min\{SV^{-2}, S^3 NV^{-6}\}\right)(\log (2NT))^{6}.
\] 
\end{lemma}

\begin{lemma}\label{mean value theorem}
Let $A(s)$ be as in Lemma \ref{large value theorem}. Then 
\[
\int_{-T}^T |A(it)|^2dt =(T+O(N))\sum_{n\leq N}|a_n|^2.
\]
\end{lemma}
\begin{lemma}\label{change of variable}
If $F: \mathbb{R}\to \mathbb{C}$ is a square-integrable function and $H\leq X$, then
\[
\int_X^{2X}|F(x+H)-F(x)|^2 dx \ll \sup_{\theta\in \big[\frac{H}{3X}, \frac{3H}{X}\big]}\int_X^{3X}|F(u+\theta u)-F(u)|^2 du.
\]
\end{lemma}
\begin{proof}[Proof of Proposition \ref{19.04.22.6:26}]

To estimate $\mathcal{K}_{h, z}(X; H)$, we partition the interval $[z^{1/k}, (2X)^{1/k}]$ into dyadic intervals based on the size of $d$. Note that, the condition $H>X^{\epsilon}$ is used to overcome $(\log X)$-loss from dyadic decomposition. Therefore, it is enough to establish a bound for 
\begin{align}\label{dyadic beyond z barrier}
\frac{1}{X}\int_X^{2X} \bigg|\sum_{\substack{x<nd^k\leq x+H\\ d\sim D}}h(d)-H\sum_{d\thicksim D}\frac{h(d)}{d^k}\bigg|^2 dx
\end{align}
for each $D\in [z^{1/k}, (2X)^{1/k}]$.
Let us look at
\[
A(x):=\sum_{\substack{nd^k\leq x\\ d\sim D}}h(d)-xB_h(k) \quad \mbox{and}\quad B_h(s) :=\sum_{d\sim D}\frac{h(d)}{d^s}.
\]
Using Perron's formula, we deduce that
\begin{align*}
A(e^y)=\frac{1}{2\pi i}\int_{2-i\infty}^{2+i\infty}\frac{e^{ys}}{s}\zeta(s)B_h(ks)ds - e^y B_h(k).
\end{align*} 
By shifting the contour of integration to the line $\Re(s)=1/2$, we capture the pole at $s=1$ and eliminate the contribution of the residue at this point when combined with the second term $e^y B_h(k)$. As a result, for any real $w$, we obtain that
\[
\frac{A(e^{w+x})-A(e^x)}{e^{x/2}}=\frac{1}{2\pi}\int_{-\infty}^{\infty}\frac{e^{w(1/2+it)}-1}{1/2+it}e^{itx}\zeta(1/2+it)B_h(k/2+kit)dt.
\]
The Plancherel formula gives 
\begin{align*}
\int_{0}^{\infty}|A(e^{u+w})-A(e^u)|^2\frac{du}{e^u}\ll \int_{\mathbb{R}}\bigg|\frac{e^{w(1/2+it)}-1}{1/2+it}\bigg|^2|\zeta(1/2+it)B_h(k/2+kit)|^2dt.
\end{align*}
Applying Lemma \ref{change of variable},
\begin{align*}
\frac{1}{X}\int_X^{2X}|A(x+H)-A(x)|^2 dx\ll \frac{1}{X}\int_X^{2X}|A(u(1+\theta))-A(u)|^2 du
\end{align*}
for some $\theta\in \big[\frac{H}{3X}, \frac{3H}{X}\big]$. Now we take $w$ satisfying $e^w=1+\theta$, so that $w\asymp \frac{H}{X}$.  Therefore, using a change of variables,
\begin{align}\label{after plancherel}
\frac{1}{X}\int_X^{2X}|A(x+H)-A(x)|^2 dx&\ll X\int_0^{\infty}|A(u(1+\theta))-A(u)|^2\frac{du}{u^2}\nonumber\\
&\ll  X\int_{\mathbb{R}}\bigg|\frac{e^{w(1/2+it)}-1}{1/2+it}\bigg|^2|\zeta(1/2+it)B_h(k/2+kit)|^2dt. 
\end{align}
The well-known sub-convexity bound tells that $\zeta(1/2+it)\ll |t|^{1/6}(\log |t|)^2$ for $|t|\geq 2$ and $B_h(k/2+kit)\ll D^{1-\alpha -\frac{k}{2}+\epsilon_1}$. Therefore, the integral \eqref{after plancherel} with $|t|\geq X^2$ contributes
\[
\ll X D^{1-\alpha -\frac{k}{2}} \int_{X^2}^{\infty}|t|^{-5/3+\epsilon}dt \ll  X^{-\frac13+2\epsilon}D^{1-\alpha -\frac{k}{2}}\ll X^{-\frac14}.
\]
Splitting into the regions $|t|\leq \frac{X}{H}$ and $2^\ell<|t|\leq 2^{\ell+1}$ with $\frac{X}{2H}\leq 2^\ell\leq X^2$, the integral \eqref{after plancherel} is bounded by
\begin{align}\label{product of l-functions}
\ll H\bigg(\sup_{X/H\leq T\leq X^2}\frac{1}{T}\int_{|t|\leq T}|\zeta(1/2+it)B_h(k/2+kit)|^2 dt\bigg)+O(X^{-\frac14}).
\end{align}
\medskip
\noindent{\it Case $1$.}\label{case3} Assume that the Lindel\"{o}f hypothesis is true. Let $0\leq \alpha<\frac12$ and $\delta >0$ be sufficiently small. Using Lemma \ref{mean value theorem}, we get that \eqref{product of l-functions} is bounded above by 
\[
\ll  \Big(\frac{ H}{D^{k+2\alpha - 1}}+\frac{H}{TD^{k+2\alpha-2}}\Big)T^\delta D^{2\epsilon_1}.
\] 
As both $\delta$ and $\epsilon_1$ are sufficiently small, we have $T^\delta D^{2\epsilon_1} \ll H^{\epsilon /2k}$. Thus, under the ranges $X/H\leq T\leq X^2$, $D\in [z^{1/k}, (2X)^{1/k}]$ and $z\geq H^{1+\epsilon}$ with the hypothesis $H\leq X^{\frac{k}{k+1}-\epsilon}$, the expression in \eqref{product of l-functions} is dominated by
\[
 \Big(\frac{ H}{D^{k+2\alpha-1}}+\frac{H^2}{XD^{k+2\alpha-2}}\Big) H^{\frac{\epsilon}{2k}} \ll \Big(H^{\frac{1-2\alpha}{k}-\frac{2k+4\alpha-3}{2k}\epsilon}+H^{\frac{1-2\alpha}{k}- \frac{4k^2+4k\alpha-k+2}{2k^2}\epsilon}\Big)  \ll H^{\frac{1-2\alpha}{k}-\frac{2k+4\alpha-3}{2k}\epsilon}.
\]

Now, let us consider the case where $\frac12<\alpha< 1$. Similar to the above estimates, we can show that \eqref{product of l-functions} $\ll H^{-\epsilon/2}$, given that $z\geq H^{1+\epsilon}$ and $H\leq X^{\frac{k}{k-2\alpha+2}-\epsilon}$.

\medskip
\noindent{\it Case $2$.}
 We now consider the unconditional case.

\medskip
\noindent{\it Subcase $2.1$.}
Let $k\geq 2$ and $0\leq \alpha<\frac12$. Consider the set $S(D)$ defined as
\[
S(D)=\{t\in [-T, T]: \, |B_h(k/2+kit)|\leq f(D)\},
\] 
where $f(D)$ is a function of $D$ to be chosen later.
 Applying the Cauchy--Schwarz inequality and the fourth moment of $\zeta(s)$, the contribution of $S(D)$ to \eqref{product of l-functions} is bounded by 
\[
\ll \frac{H}{T}\bigg(\int_{|t|\leq T} |\zeta(1/2+it)|^4 dt\bigg)^{1/2}\bigg(\int_{S(D)}|B_h(k/2+kit)|^4 dt\bigg)^{1/2}\ll H {(\log{T})}^2 f(D)^2\ll H^{\frac{1-2\alpha}{k}-\epsilon}
\]
provided that $f(D)\ll H^{\frac{1-2\alpha-k}{2k}-\epsilon}$ and using trivial bound $S(D)\ll T$. Since $D\geq z^{1/k}\geq H^{1/k+\epsilon/k}$, we choose $f(D)=D^{\frac{1-2\alpha-k}{2}- \epsilon}$.

Using the definition of $B_h$ and the bound  $h(d)\ll_{\epsilon_1} d^{\epsilon_1-\alpha}$, there exists some $c_{\epsilon_1}>0$ such that $$|B_h(k/2+ikt)|\leq (c_{\epsilon_1}D)^{1-\alpha-\frac{k}{2}+\epsilon_1}.$$ 

Let $V\in [f(D), (c_{\epsilon_1}D)^{1-\alpha-\frac{k}{2}+\epsilon_1}]$ and consider the following set:
\[
S(D)^\perp= \{t\in [-T, T]:\, V\leq |B_h(k/2+kit)|< 2V\}.
\]
Thus, it suffices to show that for each $V\in [f(D), (c_{\epsilon_1}D)^{1-\alpha-\frac{k}{2}+\epsilon_1}]$ and $T\in [X/H, X^2]$,
\[
\frac{H}{T}V^2\int_{S(D)^\perp}|\zeta(1/2+it)|^2 dt \ll H^{\frac{1-2\alpha}{k}- \epsilon}
\]
holds.
Additionally, a straightforward estimation shows that $$\sum_{d\sim D}\frac{|h(d)|^2}{d^k}\ll \frac{1}{D^{k+2\alpha-1-2\epsilon_1}}.$$
Thus Lemma \ref{large value theorem} gives 
\[
|S(D)^\perp|\ll \left(\frac{1}{D^{k+2\alpha-2}}V^{-2}+T\min\Big\{\frac{1}{D^{k+2\alpha-1}}V^{-2}, \frac{1}{D^{3k+6\alpha-4}}V^{-6}\Big\}\right)D^{8\epsilon_1}.
\]

We begin by bounding the first term of $|S(D)^\perp|$. By using the sub-convexity bound $\zeta(1/2+it)\ll |t|^{13/84+\epsilon}$, established by J. Bourgain~\cite{JB}, we obtain that when $T\in [X/H, X^2]$ and $D\in [z^{1/k}, (2X)^{1/k}]$, 
\begin{align}\label{27/06/23}
\frac{H}{T}V^2\int_{S(D)^\perp}|\zeta(1/2+it)|^2 dt \ll \frac{HD^{8\epsilon_1}}{T^{29/42-\epsilon}D^{k+2\alpha-2}}.
\end{align}
Therefore, \eqref{27/06/23}$ \ll H^{\frac{1-2\alpha}{k}-\frac{\epsilon}{2}}$, since $H\leq X^{\frac{29k}{29k+42}-\epsilon}$ and $\epsilon_1$ is sufficiently small.

Now, we analyze the second term in the above bound of $|S(D)^\perp|$. Applying the Cauchy--Schwarz inequality and the fourth moment $\zeta$ (see \cite[eq. (7.6.1)]{Titch} ), we deduce that
\begin{align}\label{05/06/2026}
\frac{H}{T}V^2\int_{S(D)^\perp}|\zeta(1/2+it)|^2 dt&\ll \frac{H}{T}V^2 |S(D)^\perp|^{1/2}\bigg(\int_{|t|\leq T}|\zeta(1/2+it)|^4 dt \bigg)^{1/2}\nonumber\\
& \ll HV^2\min\Big\{\frac{V^{-2}}{D^{k+2\alpha-1}}, \frac{V^{-6}}{D^{3k+6\alpha-4}}\Big\}^{1/2}D^{4\epsilon_1} \nonumber\\
& \ll H \frac{V^{1/2}}{D^{(k+2\alpha-1)/4}} \frac{V^{-1/2}}{D^{(3k+6\alpha-4)/4}} D^{4\epsilon_1}.
\end{align}
Therefore, \eqref{05/06/2026} $\ll H^{\frac{1-2\alpha}{k}-\epsilon/2}$ provided that $z\geq H^{\frac{k+2\alpha-1}{k+2\alpha-5/4}+\epsilon}$. Note that $\mathcal{K}_{h, z}(X; H)$ is bounded above by the dyadic expression in \eqref{dyadic beyond z barrier} multiplied by $\log{H}$. Hence, we have $\mathcal{K}_{h,z}(X; H)\ll H^{\frac{1-2\alpha}{k}-\frac{\epsilon}{3}}$.

\medskip
\noindent{\it Subcase $2.2$.}
Assume that $k \geq 1$ and $\frac{1}{2}<\alpha<2$. As in Subcase $2.1$, the result follows by comparing the bounds obtained in \eqref{27/06/23} and \eqref{05/06/2026} with $O(H^{-\epsilon/3})$, provided that provided that $H\leq X^{\frac{29k}{29k+84(1-\alpha)}-\epsilon}$ and $z\geq H^{\frac{k}{k+2\alpha-5/4}+\epsilon}$.
\end{proof}

\section{Proof of Proposition~\ref{prop5} and Proposition~\ref{prop6}}\label{section 7}
This section presents the proofs of propositions concerning the discrete variance. 
\subsubsection*{Preparation of proofs} By expanding the square to $\Var_f^{\mathcal{D}}(X; H)$  and using \eqref{mean value estimate}, we obtain 
\begin{align}\label{T(H)}
\Var_f^{\mathcal{D}}(X; H)=S_1-\widetilde{c}_{h, k}^2H^2+O\left(\frac{H^2}{X}X^{\max \big\{\frac{1-\alpha}{k}+\epsilon_1, \, 0\big\}}\right),
\end{align}
where
\[
S_1:= \frac{1}{X}\sum_{j_1, j_2=1}^{H}\sum_{n\leq X}f(n+j_1)f(n+j_2).
\]

Let us denote 
\[
f_z(n):= \sum_{\substack{d^k\mid n\\ d\leq z}}h(d)\quad \text{ and } \quad f_z^c(n):= \sum_{\substack{d^k\mid n\\ d> z}}h(d)
\]
so that $f(n)=f_z(n)+f_z^c(n)$,
where $z$ is to be chosen later.
Using the Cauchy--Schwarz inequality,
\begin{align*}
\bigg|\sum_{n\leq X}f(n+j_1)f(n+j_2)-f_z(n+j_1)f_z(n+j_2)\bigg| \ll \bigg(\sum_{n\leq X}|f_z^c(n+j_1)|^2\bigg)^{\frac12}\bigg(\sum_{n\leq X}|f_z(n+j_2)|^2\bigg)^{\frac12}.
\end{align*}
Applying the Chinese remainder theorem and $h(d)\ll \frac{1}{d^{\alpha - \epsilon_1}}$, we obtain 
\begin{align*}
\sum_{n\leq X}|f_z^c(n+j_1)|^2=\sum_{n\leq X}\bigg(\sum_{\substack{d^k\mid n+j_1\\ d> z}}h(d)\bigg)^2\ll \frac{X+j_1}{z^{2(k+\alpha-\epsilon_1-1)}}.
\end{align*}
We see that
\[
\sum_{n\leq X}|f_z(n+j_2)|^2\ll X.
\]
By combining all the aforementioned information, we deduce that
\begin{align*}
S_1=\frac{1}{X}\sum_{j_1, j_2=1}^{H}\sum_{n\leq X}f_z(n+j_1)f_z(n+j_2)+O\left(H^2z^{-(k+\alpha-\epsilon_1-1)}\right).
\end{align*}
Furthermore, using the Chinese remainder theorem again, we obtain
\begin{align*}
\sum_{n\leq X}f_z(n+j_1)f_z(n+j_2)&=X\sum_{d_1, d_2\geq 1}\frac{h(d_1)h(d_2)}{[d_1, d_2]^k}\varrho(j_1, j_2, d_1, d_2)\\
&+  + O\bigg(z^{2(1-\alpha+\epsilon_1)} + X\sum_{d_1>z; d_2\geq 1}\frac{|h(d_1)h(d_2)|}{d_1^k d_2^k}\bigg),
\end{align*}
where $\varrho(j_1, j_2, d_1, d_2)$ represents the number of solutions to the congruence system $n\equiv -j_1 \pmod{d_1^k}$ and $n\equiv -j_2 \pmod{d_2^k}$.
The error term is bounded by 
$$\ll  z^{2(1-\alpha+\epsilon_1)}+Xz^{-(k+\alpha-\epsilon_1-1)}.$$

By choosing $z=X^{\frac{1}{k-\alpha+1+\epsilon_1}},$ 
\begin{equation}\label{S1}
S_1 =\sum_{j_1, j_2=1}^{H}A(j_1, j_2) + O\left(H^2X^{-\frac{k+\alpha-1-\epsilon_1}{k-\alpha+1+\epsilon_1}}\right),
\end{equation}
where 
\[
A(j_1, j_2):= \sum_{d_1, d_2\geq 1}\frac{h(d_1)h(d_2)}{[d_1, d_2]^k}\varrho(j_1, j_2, d_1, d_2).
\] 

In order to state our next lemma, we will use the following notations. 
\begin{align*}
Q_h(s)= \begin{cases}\prod_{p}\left(1+\frac{p^k h(p)^2}{p^s(p^k+2h(p))}\right) & \mbox{ if }  h\in \mathcal{M}_{\alpha, \beta},\\
\prod_{p}\left(1-\frac{h(p)^2}{p^s}\right)^{-1} &\mbox{ if }  h\in \mathcal{G}_{\alpha, \beta},
\end{cases}
\quad \text{ and } \quad
e_h=\begin{cases}
\prod_{p}\left(1+\frac{2h(p)}{p^k}\right)   &\mbox{ if }  h\in \mathcal{M}_{\alpha, \beta},\\
\prod_p\frac{p^k+h(p)}{p^k-h(p)}   &\mbox{ if }  h\in \mathcal{G}_{\alpha, \beta}.
\end{cases}
\end{align*}
 \begin{lemma}\label{lemH}
 Let $k\geq 1$, and $0<\gamma <1/2$ with $0\leq \alpha <2$. We have 
 \begin{align*}
 R(H):=\sum_{j_1=1}^{H}\sum_{j_2=1}^{H} A(j_1,\, j_2) = 
 \widetilde{c}_{h,k}^2 H^2 + 2e_hI_{h,\alpha, \gamma}(H) ,
 \end{align*}
 where 
 \begin{align}\label{I(H)}
 I_{h, \alpha, \gamma}(H) := \frac{1}{2\pi i}\int_{(-\gamma)}H^{s+1}\chi(s)\zeta(1-s)Q_h(k+ks)\frac{ds}{s(s+1)}.
 \end{align}
 \end{lemma}
 \begin{proof}
 
Notice that 
\begin{small}
\begin{align*}
A(j_1, j_2)&=\sum_{\substack{d_1, d_2\geq 1\\ (d_1, d_2)^k\mid \, |j_1-j_2|}}\frac{h(d_1)h(d_2)}{[d_1, d_2]^k}=
\begin{cases}
e_h \displaystyle\sum_{d^k\mid \, |j_1-j_2|}\frac{h(d)^2}{d^k} \prod_{p|d}\left(1+\frac{2h(p)}{p^k}\right)^{-1}  & \mbox{ if }\, h\in \mathcal{M}_{\alpha, \beta}, \\
e_h\displaystyle\prod_{p^k\mid \, |j_1-j_2|}\left(1-\frac{h(p)^2}{p^k}\right)^{-1}  & \mbox{ if }\, h\in \mathcal{G}_{\alpha, \beta}.
\end{cases} 
\end{align*}
\end{small}
This implies  
\[
R(H)=e_h\sum_{j_1=1}^H\sum_{j_2=1}^{H}f_h(|j_1-j_2|), 
\]
where 
\[
f_h(r)=\begin{cases} \prod_{\substack{ p^k \mid  r}}\left(1+\frac{h(p)^2}{p^k+2h(p)}\right)  & \mbox{ if } h \in \mathcal{M}_{\alpha, \beta},\\
\prod_{p^k\mid r}\left(1-\frac{h(p)^2}{p^k}\right)^{-1} & \mbox{ if } h \in \mathcal{G}_{\alpha, \beta}.
\end{cases}
\]
By separating the diagonal and off-diagonal terms, we have
\[
R(H)=e_hf_h(0) H +2e_h\sum_{n=1}^H  (H-n)f_h(n)=e_hf_h(0) H +2e_h \sum_{n=1}^{H-1}\sum_{l=1}^n f_h(l).
\]
We can rewrite the above double sum as 
\[
S(H):=\sum_{n=1}^{H-1}\sum_{l=1}^n f_h(l)=\sum_{n=1}^{H-1}\sum_{l=1}^n \sum_{d^k \, \mid \, l}\frac{\nu_h(d)}{d^k}.
\]
Here $\nu_h$ is a multiplicative function defined as follows:
if $h \in \mathcal{G}_{\alpha, \beta}$, then $\nu_h(d) = h^2(d)$.
For $h \in \mathcal{M}_{\alpha, \beta}$, it is defined for prime powers as
\begin{align*}
\nu_h(p^t)=\left\{
 	\begin{array}
 	[c]{ll}
 	\frac{p^t h(p)^2}{p^t+2h(p)}  \,& \text{ if   } \,  t=1, 
 	\\
 	\hspace{2mm}
0 & \, \text{ if } \, t\geq 2.
 	\end{array}
 	\right.
\end{align*}
Then the Dirichlet series $\sum_{d=1}^{\infty} \frac{\nu_h(d)}{d^s}$ is essentially $Q_h(s)$. Following Hall [\cite{Hall}, Lemma $1$], we deduce that for $\sigma>1$,
\begin{equation}\label{S}
S(H)=\frac{1}{2\pi i}\int_{(\sigma)}H^{s+1}\zeta(s)Q_h(k+ks)\frac{ds}{s(s+1)}.
\end{equation}
 The integrand is absolutely convergent for all $\Re(s)>-\frac12$  except for simple poles at $s= 0$ and $s=1$. We move the line of integration to $\Re(s)=\gamma$ with $-\frac12<\gamma<0$. Using the Cauchy residue theorem and the functional equation of $\zeta$, we obtain
\begin{equation}\label{S(H)}
S(H)=\frac{H^2}{2}Q_h(2k)-\frac{H}{2}Q_h(k)+\frac{1}{2\pi i}\int_{(\gamma)}H^{s+1}\chi(s)\zeta(1-s)Q_h(k+ks)\frac{ds}{s(s+1)}.
\end{equation}
Since $Q_h(k)=f_h(0)$ and $e_hQ_h(2k)=\widetilde{c}_{h,k}^2$, the result follows.
 \end{proof}

\begin{lemma}\label{newlem3}
	Let $f\in \mathcal{F}_{\alpha, \beta, k}$ with $0 \leq \alpha <\frac{1}{2}$ and $0<\epsilon<\frac{3(1-2\alpha)}{k(3+k\beta^2)}$ be given. Recall the definition of $I_{h,\alpha, \gamma}(H)$ from Lemma~\ref{lemH}. Then we have 
	\begin{align*}
	I_{h,\alpha, \gamma}(H) = \mathrm{Res}_{s=\frac{1-2\alpha-k}{k}}H^{1+s}\chi(s) \zeta(1-s)Q_h(k+ks)\frac{1}{s(s+1)}+O\left(H^{\frac{1-2\alpha}{k}-\epsilon}\right).
	\end{align*}
\end{lemma}
\begin{proof}
	
	Let $\delta>0$ be optimized later. The function $Q_h(k+ks)$ exhibits a pole of order $\beta^2$ at $s=-1+\frac{1-2\alpha}{k}$. By relocating the line of integration to $\Re(s)=-\delta$, located to the left of the pole, we obtain
	\begin{align*}
	I_{h,\alpha, \gamma}(H) =\mathrm{Res}_{s=-1+\frac{1-2\alpha}{k}}H^{1+s}\chi(s) \zeta(1-s)Q_h(k+ks)\frac{1}{s(s+1)}+O\left(|I_1|+|I_2|+|I_3|\right),
	\end{align*}
	where
	\begin{small}
	\begin{align*}
	&I_1= \frac{1}{2\pi i}\int_{-\gamma+iT}^{-\gamma+i \infty}H^{s+1}\chi(s) \zeta(1-s)Q_h(k+ks)\frac{ds}{s(s+1)}, \\
	& I_2= \frac{1}{2\pi i}\int_{-\gamma+iT}^{-\delta+iT}H^{s+1}\chi(s) \zeta(1-s)Q_h(k+ks)\frac{ds}{s(s+1)},\\
	& I_3= \frac{1}{2\pi i}\int_{-\delta}^{-\delta+iT}H^{s+1}\chi(s) \zeta(1-s)Q_h(k+ks)\frac{ds}{s(s+1)}.
	\end{align*}
	\end{small}
	It is known that (\cite{GI}, page 193) for $\sigma\geq \frac{1}{2},\, t\geq 2$,
	\begin{align}\label{convexity bound}
	\zeta(\sigma+it)\ll (1+|t|)^{\frac{1-\sigma}{3}}\log t.
	\end{align}
Furthermore, for $-1 \leq \sigma \leq 2$ and $t \geq 2$, applying Stirling's formula to $\Gamma(s)$ yields
	\begin{small}
	\begin{align}\label{Bound for gamma factor}
	\chi(s)=\left(\frac{2\pi}{t}\right)^{\sigma+it -1/2}e^{i(t+\pi/ 4)}\left(1+O(1/t)\right).
	\end{align}
	\end{small}
	\subsubsection*{Evaluation of $I_2$ and $I_3$} 
	Applying \eqref{convexity bound} and \eqref{Bound for gamma factor}, we obtain
	\begin{small}
	\[
	I_2\ll \int_{-\delta}^{-\gamma}\frac{H^{1+\sigma}T^{1/2-\sigma}T^{\frac{1-2\alpha-k-k\sigma}{3}\beta^2}}{T^2}\log T \, d\sigma\ll HT^{-\frac{3}{2}+\frac{1-2\alpha-k}{3}\beta^2}\log T\max_{\gamma\leq \sigma\leq \delta}\left(\frac{H}{T^{1+\frac{k\beta^2}{3}}}\right)^{-\sigma}
	\] 
	\end{small}
	and \begin{small}
	\begin{align*}
	I_3\ll H^{1-\delta}+ H^{1-\delta}\int_{2}^{T}\frac{t^{1/2+\delta}t^{\frac{1-2\alpha-k+k\delta}{3}\beta^2}}{t^2}\log t \, dt \ll H^{1-\delta}+ H^{1-\delta}T^{-\frac{1}{2}+\delta +\frac{1-2\alpha-k+k\delta}{3}\beta^2}\log T.
	\end{align*}
	\end{small}
	\subsubsection*{Evaluation of $I_1$} Following Hall (\cite{Hall}, pages~11--12), we derive that for $T=H^{\frac{k}{1-2\alpha}\epsilon}$,
	\begin{small}
	\[
	I_1\ll \left(\frac{H}{T}\right)^{\frac{1-2\alpha}{k}}\ll H^{\frac{1-2\alpha}{k}-\epsilon},
	\]
	\end{small}
By choosing $T$ as above and $\gamma=1-\frac{1-2\alpha}{k}-\epsilon'$, where $\epsilon'>0$, we obtain $I_2\ll H^{\frac{1-2\alpha}{k}-\epsilon}$, under the condition that $\epsilon'<\frac{k\epsilon}{2\left(1-2\alpha-k\epsilon(1+k\beta^2/3)\right)}$ and $\epsilon<\frac{3(1-2\alpha)}{k(3+k\beta^2)}$.
Furthermore, if we choose $\delta=1-\frac{1-2\alpha}{k}+\epsilon'$, where $\epsilon'=\frac{k\epsilon}{2\left(1-2\alpha-k\epsilon(1+k\beta^2/3)\right)}$ and $\log{T}\ll H^{\frac{k^2\beta^2\epsilon\epsilon'}{1-2\alpha}}$, then we deduce that
	\[
I_3\ll H^{\frac{1-2\alpha}{k}-\epsilon'}+ H^{\frac{1-2\alpha}{k}-\epsilon'}H^{\frac{k}{1-2\alpha}\left(\frac{1}{2}-\frac{1-2\alpha}{k}+\epsilon' +\frac{k\beta^2}{3}\epsilon'\right)\epsilon}\ll H^{\frac{1-2\alpha}{k}-\epsilon}.
\]
	This completes the proof.
\end{proof}
 
 \begin{lemma}\label{connect between sum and infinite integral}
 Let $\frac{1}{2} < \alpha <2 $ and $c_{h,k}(H)$ be given in \eqref{C_h,k}.
 For any $H\geq 1, k\geq 1$ and $0<\gamma<1/2$, we have
 \begin{align*}
c_{h, k}(H)=2e_h I_{h,\alpha, \gamma}(H).
 \end{align*}
 \end{lemma}
 
 \begin{proof}
 Assume that $h\in \mathcal{M}_{\alpha, \beta}$. 
We write
 \[
 \left\lbrace\frac{H}{d^k}\right\rbrace- \left\lbrace\frac{H}{d^k}\right\rbrace^2=\frac{H}{d^k}-\left(\frac{H}{d^k}\right)^2+2\left(\frac{H}{d^k}\left[\frac{H}{d^k}\right]-\frac12 \left[\frac{H}{d^k}\right]\left(\left[\frac{H}{d^k}\right]+1\right)\right).
 \]
 Note that for $x>0$ and $\sigma>1$, one gets the following identity (see \cite[Page 10]{Hall}):
 \[ x[x]-\frac{1}{2}[x]([x]+1)= \frac{1}{2\pi i}\int_{(\sigma)}x^{s+1}\zeta(s)\frac{ds}{s(s+1)}.
 \]
  Consequently, we derive 
  \begin{small}
 \begin{align*}
c_{h, k}(H) &=  \sum_{d=1}^{\infty}h^2(d)\prod_{p\nmid d}\left(1 - \frac{2h(p)}{p^k}\right)\left( \frac{H}{d^k}-\left(\frac{H}{d^k}\right)^2 + \frac{1}{\pi i}\int_{(\sigma)}(H/d^k)^{s+1}\zeta(s)\frac{ds}{s(s+1)}\right)\\
&=He_hQ_h(k)- H^{2}e_hQ_h(2k)+
\frac{e_h}{\pi i}\int_{(\sigma)}H^{s+1}Q_h(k+ks)\zeta(s)\frac{ds}{s(s+1)}.
 \end{align*}
 \end{small}
 The integral on the right-hand side is $2 e_h S(H)$, with $S(H)$ defined in \eqref{S}. Thus using \eqref{S(H)},
we obtain
\[
c_{h, k}(H)=  \frac{1}{\pi i}\int_{(-\gamma)}H^{s+1}Q_h(k+ks)\chi(s)\zeta(1-s)\frac{ds}{s(s+1)}.
\]
We can repeat this argument for the remaining classes, thereby concluding the proof.
 \end{proof}

\subsubsection*{Proof of Proposition \ref{prop6}}
By incorporating Lemma \ref{lemH} and Lemma \ref{connect between sum and infinite integral} into equation \eqref{S1} and combining the result with \eqref{T(H)}, we obtain 
 \begin{align*}
\Var_f^{\mathcal{D}}(X; H)&=S_1-\widetilde{c}_{h, k}^2H^2+O\left(\frac{H^2}{X}X^{\max \big\{\frac{1-\alpha}{k}+\epsilon_1, \, 0\big\}}\right)\\
&=c_{h, k}(H)+O\left({H^2}{X^{-\frac{k+\alpha-1-\epsilon_1}{k-\alpha+1- \epsilon_1 }}}+{H^2}X^{\max \big\{\frac{1-\alpha}{k}+\epsilon_1,\,  0\big\} -1}\right).
\end{align*}
Since $H\leq X^{\frac{k+\alpha-1}{2(k-\alpha+1)}-\epsilon}$, the above error-term is $O(H^{-\epsilon})$.

 \subsubsection*{Proof of Proposition \ref{prop5}}
Following a similar approach as in the proof of Proposition~\ref{prop6}, we substitute Lemmas~\ref{lemH} and \ref{newlem3} into \eqref{S1}, and then apply \eqref{T(H)} to obtain the desired result.

\section{Proof of main theorems and corollaries}\label{section 4}
This section begins by proving our main theorems using propositions from Section \ref{section 3}. 
\begin{proof}[Proof of Theorem \ref{main theorem for simultaneouly k-free intergers}]
	Let us first consider the case when $H > X^{\epsilon}$.
	For any $f \in \mathcal{F}_{\alpha, \beta, k}$, 
	\begin{align}\label{correlation wrt z barrier}
		\sum_{x<n\leq x+H}f(n)=\sum_{x<ad^k\leq x+H}h(d)
		=\sum_{\substack{x<ad^k\leq x+H\\ d^k\leq z}}h(d)+ \sum_{\substack{x<ad^k\leq x+H\\ d^k> z}}h(d),
	\end{align}
	where $z$ will be chosen later.
	
	Denote the integrals $\mathcal{J}_{h,z}(X; H)$ and $\mathcal{K}_{h,z}(X; H)$ by $I_{k1}$ and $I_{k2}$. By applying the Cauchy-Schwarz inequality and  \eqref{correlation wrt z barrier} we have 
	\begin{align*}
		I_k:=\frac{1}{X}\int_X^{2X}\bigg|\sum_{x<n\leq x+H}f(n)- H \sum_{d^k\leq 2X}\frac{h(d)}{d^k}\bigg|^2 dx= I_{k1}+O\left(I_{k2}+\sqrt{I_{k1} I_{k2}}\right).
	\end{align*}
	
	\subsubsection*{The class $\mathcal{F}_{\alpha, \beta, 2}$}
	For $X^{\epsilon}<H\leq X^{\frac{2(3+8\alpha)}{11(1+2\alpha)}-\epsilon}$, we choose $z$ as follows: $$z=\min\Big\{X^{\frac{1}{1-\alpha}}H^{-\frac{1+2\alpha}{2-2\alpha}-\epsilon}, H^{\frac{1-2\alpha}{2-2\alpha}- \epsilon}X^{\frac{1}{2-2\alpha}}\Big\}.$$
	The range of $H$ ensures that $z \geq H^{\frac{4(1+2\alpha)}{3+8\alpha}+\epsilon}$.
	By using Proposition \ref{below z-barrier} and Proposition \ref{19.04.22.6:26}, 
	\begin{align*}
		I_{2}=
		c_{h, 2}H^{\frac{1-2\alpha}{2}}P_{\beta^2-1}(\log H)
		+O\left(H^{\frac{1-2\alpha}{2}-\frac{\epsilon}{8}}\right).	
	\end{align*}

	\subsubsection*{For the class $\mathcal{F}_{\alpha, \beta, k}$ with $k\geq 3$.}
	Recall $\nu$ and $e(k, \alpha)$ from \eqref{10/5/14:24} and \eqref{10/5/14:18} respectively. Assume that $X^{\epsilon}<H\leq X^{e(k, \alpha)-3\epsilon}$.
	Let's consider the expression \[z=\min \Big\{X^{\frac{2}{\nu(2-\alpha)}}H^{\frac{1-2\alpha}{\nu(2-\alpha)}-\epsilon}, X^{\frac{\nu}{2(\nu-\alpha)}}H^{\frac{1-2\alpha}{2(\nu-\alpha)}-\epsilon}, X^{\frac{2(k-\nu)}{(4-\nu)k-2\nu}-\epsilon}\Big\}.\] Therefore, we get that 
	$z\geq H^{\frac{k+2\alpha-1}{k+2\alpha-5/4}+\epsilon}$. Hence, whenever $1 < \nu \leq 2$, we have
	\begin{align*}
		H\leq X^{\min\big\{f_1(\nu), \,f_2(\nu),\, f_3(\nu)\big\}},
	\end{align*}
	where
	\begin{align*}
		&f_1(\nu)=\frac{2(k+2\alpha-5/4)}{\nu (2-\alpha)(k+2\alpha-1)-(1-2\alpha)(k+2\alpha-5/4)}, \\ &f_2(\nu)= \frac{\nu(k+2\alpha-5/4)}{2(\nu-\alpha)(k+2\alpha-1)-(1-2\alpha)(k+2\alpha-5/4)},\\ 
		&f_3(\nu)=\frac{2(k-\nu )(k+2\alpha-5/4)}{(4k-(k+2)\nu )(k+2\alpha-1)}.
	\end{align*}
	It can be verified that $f_1(\nu) > f_2(\nu)$ for $1.34 < \nu \leq 2$. Thus, it is sufficient to consider the minimum of $f_2(\nu)$ and $f_3(\nu)$. The intersection point of $f_2(\nu)$ and $f_3(\nu)$ is given by \eqref{10/5/14:24}.
	
	Using Proposition \ref{below z-barrier} and Proposition \ref{19.04.22.6:26}, and rescaling $3\epsilon$ to $\epsilon$, we obtain
	\begin{align*}
		I_k=
		c_{h, k}H^{\frac{1-2\alpha}{k}}P_{\beta^2-1}(\log H)
		+O\left(H^{\frac{1-2\alpha}{k}-\frac{\epsilon}{3(k+1)}}\right).
	\end{align*}
	Next, we estimate the tail part of $\Var_f(X; H)$.
	We have 
	\begin{align*}
		\sum_{d^k\geq 2X}\frac{h(d)}{d^k}\ll X^{-\frac{k+\alpha-1-\epsilon_1}{k}}.
	\end{align*}
	Using this estimate, we obtain
	\begin{align}\label{E(k)}
		E_k:= \frac{1}{X}\int_{X}^{2X}\bigg|H \sum_{d^k\geq X}\frac{h(d)}{d^k}\bigg|^2 dx\ll H^2 X^{-\frac{2(k+\alpha-1-\epsilon_1)}{k}}.
	\end{align}
	By applying the Cauchy-Schwarz inequality, we obtain
	\[
	\Var_f(X; H)=I_k + O\left(E_k+ \sqrt{I_k\cdot E_k}\right),
	\]
	The estimates for $I_k$ and $E_k$ lead to 
	\[
	\sqrt{I_k\cdot E_k}\ll H X^{-\frac{k+\alpha-1-\epsilon_1}{k}}H^{\frac{1-2\alpha}{2k}}(\log H)^{\beta^2-1}\ll H^{\frac{1-2\alpha}{k}-\epsilon}.
	\]
	This inequality holds when $H \leq X^{\frac{2(k+\alpha-1)}{2k+2\alpha-1}-\epsilon}$. By substituting the above estimates into $\Var_f(X; H)$, we complete the proof, assuming that $H > X^{\epsilon}$.

	Now consider the case when $2 \leq H \leq X^{\epsilon}$. In this case, we note that the difference between continuous and discrete variance is small, of size $O(H/X)$. Therefore, Theorem \ref{main theorem for simultaneouly k-free intergers} follows from Proposition~\ref{prop5}. This completes the proof. 
\end{proof}

\begin{proof}[Proof of Theorem \ref{main theorem for lower growth rate}]
	Assuming $H > X^{\epsilon}$, let $f \in \mathcal{F}_{\alpha, \beta, k}$ with $1/2 < \alpha < 2$. We will follow the proof of the Theorem \ref{main theorem for simultaneouly k-free intergers}. We consider two cases:
	
	\noindent
	\textbf{Case $1$:} Assume that $k=1$. We begin with the case $\alpha \in \left(\frac{1}{2},1\right)$; the remaining range of $\alpha$ can be treated similarly. According to the notation in the proof of Theorem~\ref{main theorem for simultaneouly k-free intergers}, and using Proposition~\ref{proposition for lower growth family}, we obtain that if
	$
	z \leq X^{\frac{1}{2(2-\alpha)}} H^{\frac{1}{2(2-\alpha)}-\epsilon},
	$
	then
	\[
	I_{11}=c_{h, 1}(H)+O(H^{-\epsilon_{\alpha, 1}}).
	\]
	On the other hand, Proposition~\ref{19.04.22.6:26} yields that for $X^{\epsilon}\leq H\leq X^{\frac{29}{113-84 \alpha}-\epsilon}$, and $z\geq H^{\max\big\{1, \frac{4}{8\alpha-1}\big\}+\epsilon}$, we get
	\[
	I_{12}\ll H^{-\epsilon/3}.
	\]
	By comparing the admissible ranges for $z$, we conclude that whenever
	\[
	X^{\epsilon}\leq H\leq X^{\min\big\{\frac{29}{113-84\alpha}, \frac{8\alpha - 1}{17-16\alpha}\big\}-\epsilon},
	\]
	we obtain
	\[
	I_1=c_{h, 1}(H)+O(H^{-\epsilon_{\alpha, 1}}).
	\]
The computation of the tail part proceeds exactly as in the proof of Theorem \ref{main theorem for simultaneouly k-free intergers}, and hence
\[
\Var_f(X; H)
=
I_1 + O\!\left(E_1 + \sqrt{I_1 \cdot E_1}\right),
\]
where from \eqref{E(k)}, we have $E_1 \ll H^2 X^{-2(\alpha-\epsilon_1)}$.
Therefore,
\[
\Var_f(X; H)
=
c_{h,1}(H) + O\!\left(H^{-\epsilon_{\alpha,1}}\right),
\]
provided that
\[
X^{\epsilon} \leq H \leq X^{\min\left\{\alpha, \frac{29}{113-84\alpha}, \frac{8\alpha - 1}{17-16\alpha}\right\}-\epsilon}\implies X^{\epsilon} \leq H \leq X^{\min\left\{\frac{29}{113-84\alpha}, \frac{8\alpha - 1}{17-16\alpha}\right\}-\epsilon}.
\]

	\noindent
	\textbf{Case $2$:} Assume that $k\neq 1$. The proof for the rest of the values of $k$ follow the same arguments as in Case $1$ and hence omit the details. 
	
	\noindent
	For the range $2\leq H\leq X^{\epsilon}$, Theorem~\ref{main theorem for lower growth rate} follows directly from Proposition~\ref{prop6}. 
\end{proof}

\begin{proof}[Proof of Theorem \ref{main theorem complex case}] 
	
	Consider the case where $H> X^{\epsilon}$. It is important to note that Lemma \ref{newlem1} and Lemma \ref{E(H)} remain valid for $\beta\in \mathbb{C}$. By combining these Lemmas with Lemma \ref{complex case main lemma}, Proposition \ref{19.04.22.6:26}, and following the proof of Theorem \ref{main theorem for simultaneouly k-free intergers}, we conclude the proof. We treat the case of complex \( \beta \) separately from the integer case, as the associated error term is significantly weaker. Moreover, the Lemma~\ref{complex case main lemma} plays a crucial role in handling complex \( \beta \) by employing a method of Selberg--Delange.

	Next, suppose that $2\leq H\leq X^{\epsilon}$. For $\beta\in \mathbb{C}$, we can derive an analogue of Proposition \ref{prop5} with a weaker error term, which matches our asymptotic formula, by modifying its proof. Indeed, we can easily adapt the proof of Proposition \ref{prop5} to obtain the integral \eqref{I(H)}, where complex $\beta$ is not effective. To compute \eqref{I(H)}, we follow the method presented in the proof of Lemma \ref{complex case main lemma} and complete the proof.
\end{proof}

\begin{proof}[Proof of Corollaries \ref{k-free case}, \ref{main theorem of euler toitient} and \ref{divisor functions}]
	Corollary \ref{k-free case} is a direct application of Theorem \ref{main theorem for simultaneouly k-free intergers} with growth factors $\alpha=0$, $\beta=1$, and a given order $k$. Additionally, Corollary \ref{main theorem of euler toitient} and \ref{divisor functions} are consequences of Theorem \ref{main theorem for lower growth rate} with $\alpha=\beta=k=1$, and $\alpha$ varying while $\beta=k=1$, respectively.
\end{proof}
\begin{proof}[Proof of Theorem \ref{lower bound for euler phi}]
	We have
	\[
	c_{\zeta}(H)>\sum_{d=1}^{H}\frac{\mu^2(d)}{d^2}\left( \left\{ \frac{H}{d} \right\} - \left\{\frac{H}{d} \right\} ^2\right)+O(H^{-1})
	\]
	Since $H$ runs over the primes, we have $\left\{ \frac{H}{d} \right\}=\frac{a}{d}$ for some $(a, d)=1$ with $1\leq a < d$. By analyzing the function $f(x)=x-x^2$ for $x\in(0, 1)$, we determine that the minimum value of $f$ is attained at $\frac{1}{d}$ and $1-\frac{1}{d}$ for rational numbers of the form $\frac{a}{d}$. Therefore, we have
	\begin{align*}
		c_{\zeta}(H)& >\sum_{d=1}^{H}\frac{\mu^2(d)}{d^2}\left(\frac{1}{d}-\frac{1}{d^2}\right)+O(1/H)\\
		& = \frac{\zeta(3)}{\zeta(6)}-\frac{\zeta(4)}{\zeta(8)}+O(1/H)=0.1036\ldots +O(1/H).
	\end{align*}
\end{proof}

\section{Appendix}\label{appendix}
%\subsection{More applications}\label{applications}
As a consequence of our main results, we introduce a few additional corollaries that have been extensively studied in the existing literature.
\subsection{Euler totient function on Selberg class}
In \cite{JK}, Kaczorowski considered a general polynomial Euler product $F(s)$ of degree $d\geq 1$,
which belongs to a specific subclass of the Selberg class of $L$-functions. This particular class is defined by
\[
F(s)=\prod_p F_p(s)=\prod_p \prod_{j=1}^{d}\left(1-\frac{\alpha_j(p)}{p^s}\right)^{-1}, \quad \mbox{ for }\Re(s)>1,
\]
where $|\alpha_j(p)|\leq 1$ for all $p$ and $1\leq j\leq d$.
Additionally, the associated Euler totient function $\phi(n, F)$ is defined as
\[
\phi(n, F)=n\prod_{p\mid n}F_p(1)^{-1}=n\sum_{m\mid n}\frac{\mu(m)r(m)}{m},
\]
where 
$
r(m)=\prod_{p\mid m}p\left(1-F_p(1)^{-1}\right).
$
Through his work, Kaczorowski established that the mean value of $\phi(n, F)$ is equal to $C(F)x+O\big((\log 2x)^d\big),$
where $
C(F)=\prod_{p}\left(1-\frac{1-F_p(1)^{-1}}{p}\right)$.
Moreover, by considering the usual zero-free region for $F$, the aforementioned result is extended on average, together with its continuous variance:
\begin{small}
	\begin{align*}
	\frac{1}{X}\int_{X}^{2X}\bigg| \sum_{n\leq x}\frac{\phi(n, F)}{n}-2C(F)\, x\bigg|^2dx=\frac{7}{6\pi^2}\prod_p \left(1+\frac{|r(p)-1|^2}{p^2-1}\right)+O\left( \exp\left(-c\sqrt{\log X}\right)\right),
	\end{align*}
\end{small}
It is important to note that the function $r$ is multiplicative and satisfies $r(m)\ll m^{\epsilon}$ for any sufficiently small $\epsilon>0$. Consequently, as a direct consequence of Theorem \ref{main theorem for lower growth rate}, the following variance for $\frac{\phi(n, F)}{n}$ is obtained.
\begin{corollary}\label{1.6}
	Let $\epsilon>0$ be given. Then for $2\leq H\leq X^{1-\epsilon}$, we have 
	\begin{small}
		\begin{equation*}
		\Var_{\frac{\phi(n, F)}{n}}(X; H) = \sum_{d=1}^{\infty}\frac{\mu^2(d)r^2(d)}{d^2}\prod_{p\nmid d}\left(1 + \frac{2(1-F_p(1)^{-1})}{p}\right)\left( \left\{ \frac{H}{d} \right\} - \left\{\frac{H}{d} \right\} ^2\right) +O(H^{-\epsilon}).
		\end{equation*}
	\end{small}
\end{corollary}
\subsubsection{Euler function twisted by Dirichlet character}

Let $F(s)=L(s,\psi)$, where $\psi$ is a real non-principal Dirichlet character modulo $q$. Consider the function $\phi(n, \psi)$ given by
\[
\phi(n, \psi):=n\prod_{p\mid n}\left(1-\frac{\psi(p)}{p}\right).
\]
The Omega results for $\sum_{n\leq x}\phi(n, \psi)-\frac{x^2}{2L(2,\psi)}$ were investigated by Kaczorowski and Wiertelak \cite{JKW}, and Kaczorowski \cite{JK} further examined the mean square for the same expression. Moreover, applying Corollary \ref{1.6}, we can also derive the variance for $\frac{\phi(n, \psi)}{n}$ in short intervals.

\subsection{M\"{o}bius function of order $k$}
For $r\geq 1$ and $k\geq 2$, the M\"{o}bius function of order $k$ was introduced by T. M. Apostol \cite{Apostol}. It is defined by 
\begin{align*}
\widehat{\mu}_k(n)=
\begin{cases}
1 & \mbox{ if } \, n=1,\\
0 & \mbox{ if } \, p^{k+1}\mid n \, \text{ for some prime } p,\\
(-1)^r & \mbox{ if }\, n=p_1^k \ldots p_r^k\prod_{i>r}p_i^{a_i},\, 0\leq a_i<k,\\
1 & \mbox{ otherwise}.
\end{cases}
\end{align*}
When specializing to the case $k=1$, we recover the classical M\"{o}bius function. To study the average behavior of $\widehat{\mu}_k(n)$, it suffices to consider the following mean value:
\[
F_r(x):=\sum_{n\leq x}\widehat{\mu}_{k-1}(n)\widehat{\mu}_{k-1}(r^{k-1}n)=\mu(r)\sum_{\substack{n\leq x\\ (n, r)=1}}\mu_{k}(n).
\]
Note that, 
\[
\sum_{n\leq x}\widehat{\mu}_k(n)=\sum_{r\leq x^{1/k}}F_r(x/r^k).
\]
The subsequent corollary directly follows from Theorem \ref{main theorem for simultaneouly k-free intergers}. 
\begin{corollary}
	Let $\epsilon, e(k, 0)$ and $g(k)$ be as in Theorem \ref{main theorem for simultaneouly k-free intergers}. Let $H\leq X^{e(k, 0)-\epsilon}$. Then \begin{align*}
	\frac{1}{X}\int_X^{2X}\bigg|F_r(x+H)-F_r(x)- \frac{\phi(r) r^{k-1} \mu(r)}{\zeta(k)J_k(r)}H\bigg|^2 dx =  c_{ k, r}H^{\frac{1}{k}} +O\left(H^{\frac{1}{k}- \frac{\epsilon}{3(k+1)}}\right),
	\end{align*}
	where
	\begin{equation*}
	c_{k, r}=  -\frac{\mu^2(r)}{2\pi^2}\chi{\Big(\frac{k+1}{k}\Big)}\zeta{\Big(2-\frac{1}{k}\Big)}\prod_{p}\bigg( 1-\frac{1}{p^2} - \frac{2(p-1)}{p^{k+1}}\bigg) \prod_{p\mid r}\left(1-\frac{2}{p^k}+\frac{1}{p}\right)^{-1}.
	\end{equation*}
	%Assuming the Lindel\"{o}f hypothesis, the above estimate holds for $H\leq X^{g(k)-\epsilon}$.
\end{corollary}

Furthermore, we can use Theorem \ref{main theorem for simultaneouly k-free intergers} to derive the variance of an arithmetic function associated with counting prime factors. This serves as an example in which we use $\alpha=0$ and $\beta=2$.
\begin{corollary}{\label{modified w(n)}}
	Let $\epsilon, e(k, 0)$ and $g(k)$ be as in Theorem \ref{main theorem for simultaneouly k-free intergers}. For $2\leq H\leq X^{e(k, 0)-\epsilon}$, 
	\begin{align*}
	\frac{1}{X}\int_X^{2X}\bigg|\sum_{x<n\leq x+H}(-1)^{\#\{p:\, p^k\mid n\}}- \prod_{p}\left(1-\frac{2}{p^k}\right)H\bigg|^2 dx =  c_{ k}H^{\frac{1}{k}}P_3(\log H) + O\left(H^{\frac{1}{k}- \frac{\epsilon}{3(k+1)}}\right),
	\end{align*}
	where 
	\begin{small}
		\begin{equation*}
		c_k=  -\frac{1}{24\pi^2}\chi{\Big(\frac{k+1}{k}\Big)}\zeta{\Big(2-\frac{1}{k}\Big)}\prod_{p} \left(1-\frac{4}{p^k}+\frac{4}{p}\right)\left(1-\frac{1}{p}\right)^4.
		\end{equation*}
	\end{small}
	%Assuming the Lindel\"{o}f hypothesis, the above estimate holds for a wider range $H\leq X^{g(k)-\epsilon}$. 
	
\end{corollary}

\subsection{Extension to subfamily $\{f(n)\}_{(n, q)=1}$ for a given integer $q$}
Building upon the proof of Theorems \ref{main theorem for simultaneouly k-free intergers} and \ref{main theorem for lower growth rate}, we can generalize our result to a subfamily $\{f(n)\}_{(n, q)=1}$ for any fixed integer $q\geq 1$. This extension involves examining the mean square of 
\begin{small}
	\[
	\sum_{\substack{x<n\leq x+H\\ (n, q)=1}}f(n)- \frac{\phi(q)}{q}\bigg(\sum_{(d, q)=1}\frac{h(d)}{d^k}\bigg) H.
	\] 
\end{small}
For $\alpha\in [0,\frac12)$, the constant $c_{h, k}$ (see \eqref{constant 2}) undergoes the following modifications:
\begin{small}
	\begin{align*}
	B(s)=\prod_{p\mid q} \left( 1+\frac{2h(p)}{p^{k}} + \frac{h(p)^2}{p^{k+ ks}} \right)^{-1}\prod_{p} \left( 1+\frac{2h(p)}{p^{k}} + \frac{h(p)^2}{p^{k+ ks}} \right)\left(1-\frac{1}{p^{k+ks+2\alpha}}\right)^{\beta^2}
	\end{align*}
\end{small}
and 
\begin{small}
	\begin{align*}
	D(s)&=\prod_{p}\left(1-\frac{h(p)^2}{p^{k+ks}}\right)^{-1}\left(1-\frac{1}{p^{k+ks+2\alpha}}\right)^{\beta^2}
	\left(\frac{p^k+h(p)}{p^{k}-h(p)}\right) \prod_{p\mid q} \left(1-\frac{h^2(p)}{p^{k+ks}}\right)\left(\frac{p^k+h(p)}{p^{k}-h(p)}\right).
	\end{align*}
\end{small}
For $\alpha\in (1/2, 2)$, when $h\in \mathcal{M}_\alpha^{\mu}$, the constant $c_{h,k}(H)$ (which depends on $H$) transforms into the following constant that is also based on $q$.
\begin{small}
	\[
	c_{h,k}(H)=\sum_{\substack{d=1\\ (d, q)=1}}^{\infty}h^2(d)\prod_{p\nmid dq}\left(1 + \frac{2h(p)}{p^k}\right)\sum_{t\mid q}\mu^2(t)\prod_{\substack{p\mid q\\ p\nmid t}}\left(1-\frac{2}{p}\right)\left( \left\{\frac{H}{td^k} \right\} - \left\{\frac{H}{td^k} \right\} ^2\right).
	\]
\end{small}
A similar modification applies to $h\in \mathcal{G}_{\alpha, \beta}$.
\subsubsection{Schemmel’s totient in short intervals}\label{sch totient func}
Schemmel introduced the Schemmel's totient function $S_m(n)$ in \cite{SCH} to count the number of sets of $m$ consecutive integers, each less than $n$, that are relatively prime to $n$. The function $S_m(n)$ has a product representation as
\begin{align*}
S_m(n)=
\begin{cases}
0 & \mbox{ if } \, P^{-}(n)\leq m,\\
n\prod_{p\mid n}\left(1-\frac{m}{p}\right)=n\sum_{\substack{d\mid n}}\frac{\mu(d)m^{w(d)}}{d} & \mbox{ if } \, P^{-}(n)>m.
\end{cases}
\end{align*}
The subsequent corollary arises as a consequence of the previously discussed subfamily with $\alpha=1$.
\begin{corollary}[Sieving by the primes $\leq m$]
	Let $\epsilon>0$ be given. Assuming that $P(m)=\prod_{p\leq m}p$, for $H\leq X^{1-\epsilon}$, the following holds: 
	\begin{equation*}
	\frac{1}{X}\int_{X}^{2X}\bigg|\sum_{\substack{x<n \leq x+H\\ \left(n, P(m)\right)=1}}\frac{S_m(n)}{n}-\prod_{p\leq m}\left(1-\frac{1}{p}\right)\prod_{p\nmid P(m)}\left(1-\frac{m}{p^2}\right)\,H\bigg|^2dx = c_m(H) +O(H^{-\epsilon}),
	\end{equation*}  
	where
	\begin{small}
		\[
		c_m(H)=\sum_{\substack{d=1\\ (d, P(m))=1}}^{\infty}\frac{\mu^2(d)m^{2\omega(d)}}{d^2}\prod_{p\nmid d}\left(1 - \frac{2m}{p^2}\right)\sum_{t\mid P(m)}\prod_{\substack{p\mid P(m)\\ p\nmid t}}\left(1-\frac{2}{p}\right)\left( \left\{ \frac{H}{td} \right\} - \left\{\frac{H}{td} \right\} ^2\right).
		\]
	\end{small}
\end{corollary}
The Euler totient function has several generalizations, such as the Jordan totient function, Klee's totient function, and others, which can be found in [Chapter 3, \cite{Sandors}] and \cite{Sandors 2}. Similarly, we can determine the variance of these functions using Theorem \ref{main theorem for lower growth rate}.
\subsection{Figures related to exponents in our main results} 

We conclude the paper to presenting figures related to exponents in our main results.

\begin{figure}[hbt!]
	\includegraphics[scale=0.42]{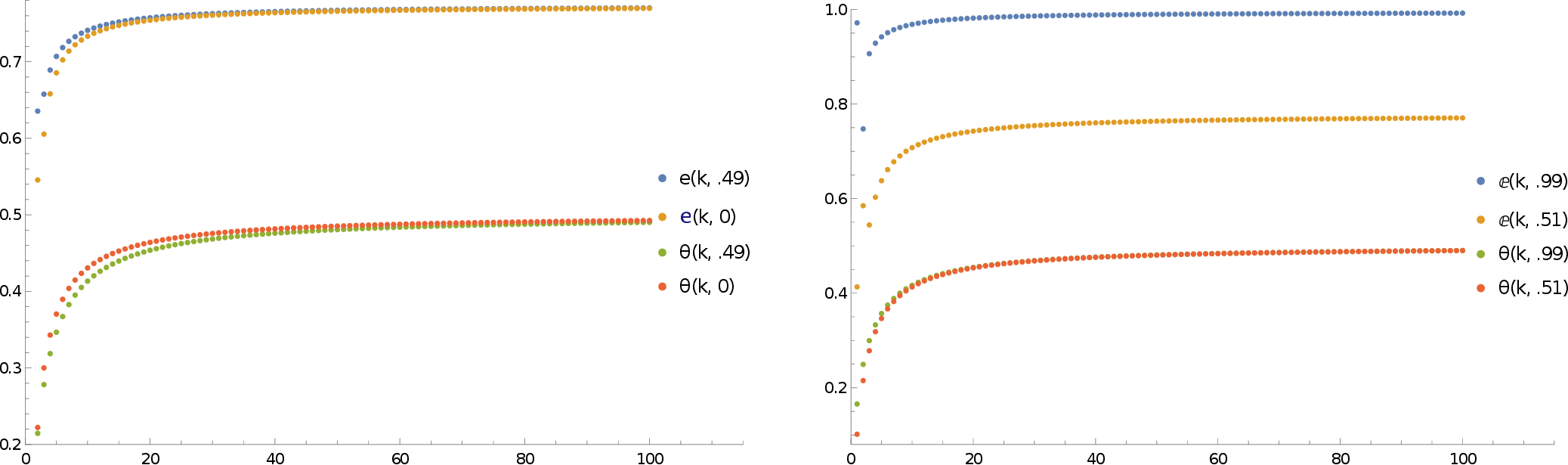}
	\caption{Comparison between exponents of continuous and discrete variance}\label{continuous and discrete exponents} 
\end{figure}

\begin{figure}[hbt!]
	\includegraphics[scale=0.50]{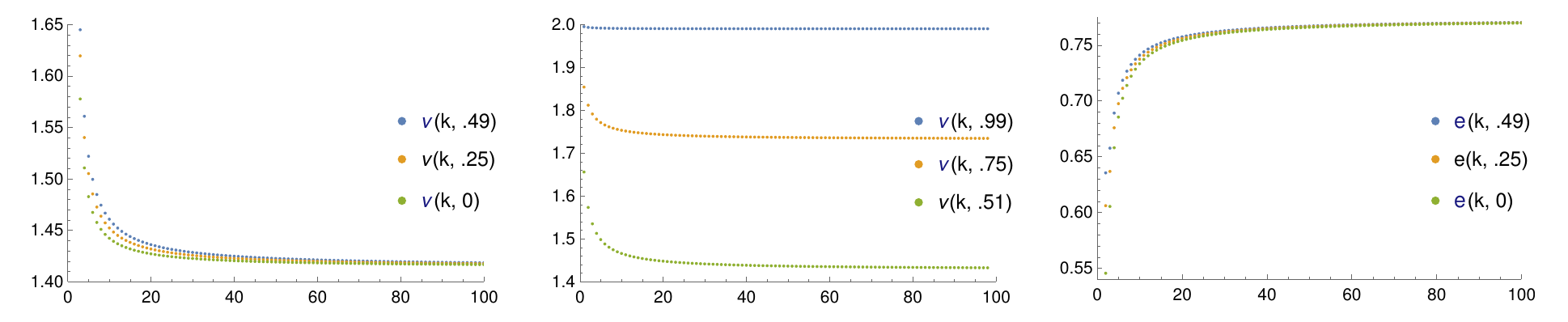}
	\caption{ Plot $\nu(k,\alpha)$ and $e(k,\alpha)$ for a few fixed $\alpha$ and $k$ vary }\label{figure 1}
	%	\label{figure_outlier}
	\includegraphics[scale=0.44]{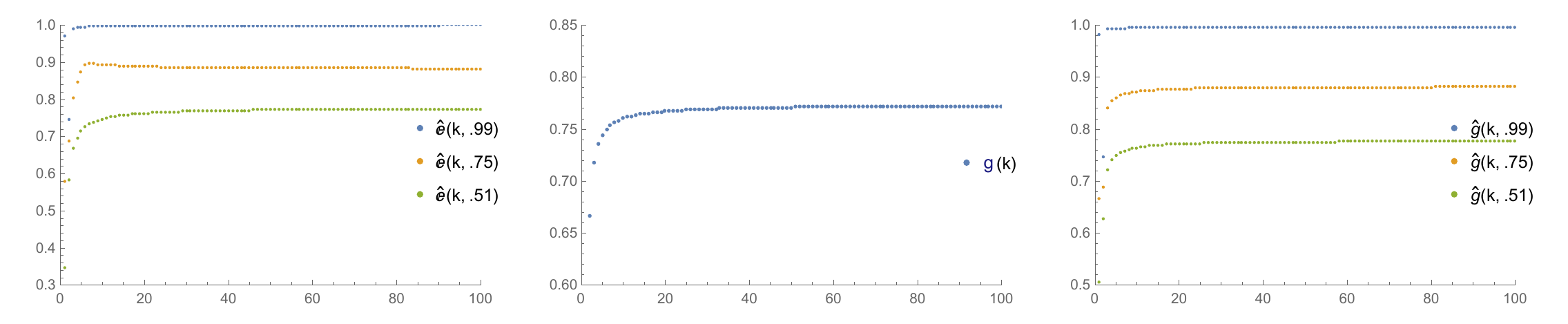}	%\includegraphics[scale=0.6]{plot1.pdf}  \quad 
	\caption{ Plot $\widehat{e}(k, \alpha)$, $g(k)$ and $\widehat{g}(k, \alpha)$ for a few fixed $\alpha$ and $k$ vary }\label{figure 2}
\end{figure}

\textbf{Data Availability Statement:} All data generated during this study are included in this article. We have no conflicts of interest to disclose.

\end{document}